\documentclass[onefignum,onetabnum]{siamonline220329}
\pdfoutput=1
\usepackage{graphicx}
\usepackage{mathtools}

\usepackage{amsmath, amssymb, xcolor}
\usepackage[T1]{fontenc}
\usepackage[utf8]{inputenc}
\usepackage{lmodern}
\usepackage{tikz}
\usepackage{pgfplots}

\definecolor{mycolor1}{rgb}{0.00000,0.44700,0.74100}
\definecolor{mycolor2}{rgb}{0.8500, 0.3250, 0.0980}
\definecolor{mycolor3}{rgb}{0.9290, 0.6940, 0.1250}
\definecolor{mycolor4}{rgb}{0.4940, 0.1840, 0.5560}
\definecolor{mycolor5}{rgb}{0.4660, 0.6740, 0.1880}

\newsiamremark{remark}{Remark}
\newsiamremark{hypothesis}{Hypothesis}
\crefname{hypothesis}{Hypothesis}{Hypotheses}
\newsiamthm{claim}{Claim}
\newsiamthm{example}{Example}

\newtheorem{observation}[theorem]{Observation}
\newtheorem{algo}{Algorithm}

\numberwithin{theorem}{section}

\headers{Bayesian Integrals on Toric Varieties}{M. Borinsky, A.-L. Sattelberger, B. Sturmfels, and S. Telen}

\title{Bayesian Integrals on Toric Varieties\thanks{
Received April 14, 2022; accepted for publication (in revised form) January 10, 2023;
\href{https://doi.org/10.1137/22M1490569}{https://doi.org/10.1137/22M1490569}
\funding{M.\ B.\ was supported by Dr.\ Max R\"{o}ssler, the Walter Haefner Foundation, and the ETH Z\"{u}rich Foundation.}}}

\author{Michael Borinsky\thanks{Institute for Theoretical Studies, ETH Z\"urich, Z\"urich
  (\email{michael.borinsky@eth-its.ethz.ch})}
\and Anna-Laura Sattelberger\thanks{MPI MiS, Leipzig and Dept.~of Mathematics, Royal Institute of Technology, Stockholm {\em (current)}
  (\email{alsat@kth.se})}
\and Bernd Sturmfels\thanks{MPI MiS, Leipzig and UC Berkeley
	  (\email{bernd@mis.mpg.de})}
\and Simon Telen\thanks{MPI MiS, Leipzig and CWI, Amsterdam {\em (current)}
(\email{simon.telen@cwi.nl})}
}

\newcommand{\PP}{\mathbb{P}}
\newcommand{\RR}{\mathbb{R}}

\newcommand{\CC}{\mathbb{C} }
\newcommand{\ZZ}{\mathbb{Z}}
\newcommand{\NN}{\mathbb{N}}
\newcommand{\Hom}{\rm Hom}

\ifpdf
\hypersetup{
  pdftitle={Bayesian Integrals on Toric Varieties},
  pdfauthor={M. Borinsky, A.-L. Sattelberger, B. Sturmfels, and S. Telen},
}
\fi

\begin{document}

\maketitle

\begin{abstract}
We explore the positive geometry of statistical models in the setting
of toric varieties. Our focus lies on models for discrete data that are
parameterized in terms of Cox coordinates. We develop a geometric theory
for computations in Bayesian statistics, such as evaluating marginal
likelihood integrals and  sampling from posterior distributions.
These are based on a tropical sampling method for evaluating
Feynman integrals in physics. We  here extend that method from 
projective spaces to arbitrary toric varieties.
\end{abstract}



\section{Introduction}\label{sec1}
Every  projective toric variety $X$ is a positive geometry \cite{arkani2017positive}. Its
canonical differential form $\Omega_X$ has poles on the toric boundary, and it
encodes probability measures on the positive part $X_{> 0}.$ Our aim is to
develop the geometry of Bayesian statistics in this toric setting.
We introduce parametric statistical models by mapping
$X_{> 0}$ into a probability simplex $\Delta_m.$
The probabilities are written in Cox coordinates on $X.$ 
We shall use the canonical measure on $X_{> 0}$ for marginal likelihood 
integrals and for sampling from the posterior distribution.

We begin with an example for the
product of three projective lines
$X = \PP^1 \times \PP^1 \times \PP^1.$ This toric threefold 
has six Cox coordinates
$ x_0,x_1,\,s_0,s_1,\, t_0,t_1.$ Each letter refers
to homogeneous coordinates on one of the three lines $\PP^1.$
We consider the model $X_{> 0} \rightarrow \Delta_m$~parameterized~by
\begin{equation}
	\label{eq:runningex1} p_{\ell} \,\, = \,\,
	\binom{m}{\ell} \frac{x_0}{x_0+x_1} \frac{s_0^{\ell} s_1^{m-\ell }}{(s_0+s_1)^m} \,+\,
	\binom{m}{\ell} \frac{x_1}{x_0+x_1} \frac{t_0^{\ell} t_1^{m-{\ell}}}{(t_0+t_1)^m} \qquad  \text{for} \quad \ell=0,1,\ldots,m.
\end{equation}
These expressions are rational functions on $X,$ positive on $X_{> 0},$
and their sum   equals~$1.$ This is the 
conditional independence model for $m$ binary random variables with $1$ binary hidden~state.
Algebraically, it represents  symmetric $2 {\times} 2 {\times} \cdots {\times} 2$ tensors of nonnegative rank~$2.$

For an intuitive understanding, 
imagine  a gambler who has three biased coins, one in each hand, and one more to decide which hand to use.
The latter coin has probabilities $x_0$ and $x_1$ for tails and heads, and this
decides whether the left hand coin (with bias $s_0,s_1$) or the right hand coin 
(with bias $t_0,t_1$) is to be used. The gambler performs $m$  coin tosses with~the chosen hand
and records the number of heads. The probability of observing $\ell$ heads equals~$p_\ell.$
A more familiar formula for this event arises by dehomogenizing via
$x_1 = x, \,x_0 = 1-x,$ etc: 
\begin{equation}
	\label{eq:runningex2}
	p_\ell \,\,=\,\, \binom{m}{\ell} \left [ (1-x)(1-s)^\ell s^{m-\ell} \,+\, x (1-t)^\ell t^{m-\ell}\, \right] 
	\qquad {\rm for} \quad \ell=0,1,\ldots,m. 
\end{equation}
As is customary in toric geometry, we identify the
positive variety $X_{> 0}$ with the open cube $(0,1)^3,$ 
which is the space of dehomogenized parameters $(x,s,t).$
At first glance, the passage from \eqref{eq:runningex2} to
\eqref{eq:runningex1} does not  change much.
It is a reparameterization of the model in $\Delta_m = \PP^m_{> 0},$ which comprises
positive Hankel matrices  that are semidefinite and have rank $\leq 2.$
For instance, for $m=4$ coin tosses, these are the $3 \times 3$ Hankel matrices shown in
\cite[Section 1]{HKS}:
$$ 
\begin{pmatrix}
	12 p_0 & 3 p_1 & 2 p_2 \\
	3 p_1 & 2 p_2 & 3 p_3 \\
	2 p_2 & 3 p_3 & 12 p_4 
\end{pmatrix} \,\,=\,\,
\frac{12}{x_0 + x_1} \begin{pmatrix} s_1^2 & t_1^2 \\ s_0 s_1 & t_0 t_1 \\ s_0^2 & t_0^2 \end{pmatrix} \!\!
\begin{pmatrix} \frac{x_0}{(s_0+s_1)^4} \!\! & 0 \! \!\smallskip \\ \! \! 0 & \!\! \frac{x_1}{(t_0+t_1)^4} \end{pmatrix} \!\!
\begin{pmatrix} s_1^2 & s_0 s_1 & s_0^2 \\ t_1^2 & t_0 t_1 & t_0^2 \end{pmatrix}.
$$

The key insight for what follows is that our toric $3$-fold
$X$ has a canonical $3$-form
\begin{equation}
	\label{eq:keyinsight}
	\Omega_X \, = \, \sum_{i=0}^1 \sum_{j=0}^1 \sum_{k=0}^1  \, (-1)^{i+j+k} \,
	\frac{{\rm d} x_i}{x_i} \wedge \frac{{\rm d} s_j}{s_j} \wedge \frac{{\rm d} t_k}{ t_k} .
\end{equation}
This gives $(X,X_{>0})$ the structure of a
positive geometry in the sense of  \cite{arkani2017positive}.
The associated representations of prior distributions on the parameter space $X_{> 0}$ offer
novel tools for Bayesian inference. For instance, suppose
our prior belief about the parameters $(x,s,t)$ in the coin model 
\eqref{eq:runningex2} is the uniform distribution on the cube~$[0,1]^3.$
Its pullback to $X$ equals
\begin{equation}
	\label{eq:unif}
	\Omega^{\rm unif}_X \,\, = \,\,
	\frac{x_0x_1s_0s_1t_0t_1}{(x_0+x_1)^2(s_0+s_1)^2(t_0+t_1)^2} \, \Omega_X . 
\end{equation}

Statistics is about data. If our gambler performs the experiment $U$
times, and $\ell$ heads were observed $u_\ell$ times, 
then $\frac{1}{U}(u_0,u_1,\ldots,u_m) \in \Delta_m$ is the empirical
distribution. The likelihood function is a rational function on
the toric variety $X,$ namely $L_u= p_0^{u_0} p_1^{u_1} \cdots p_m^{u_m}.$
The posterior distribution  is the product of
this  function times the prior distribution on
$X_{> 0}.$ 
For the prior that is uniform on $[0,1]^3 $ 
we take \eqref{eq:unif}.
The marginal likelihood integral equals
\begin{equation}
	\label{eq:mli}
	\int_{X_{> 0}} p_0^{u_0} p_1^{u_1} \cdots p_m^{u_m} \,\Omega^{\rm unif}_X .
\end{equation}
Two important tasks in Bayesian statistics \cite{Gelman}
are evaluating the
integral \eqref{eq:mli} and sampling from the posterior distribution.
See \cite{LSZ} and \cite[Section 5.5]{Sullivant} 
for points of entry from an algebraic perspective.
In this paper, we explore these tasks using
toric and tropical geometry.

We shall study our  statistical problem in the following  algebraic framework.
Let $f$ and $g$ be homogeneous polynomials of the same degree
in the Cox coordinates on an $n$-dimensional toric variety $X.$
We assume that all coefficients in $f$ and $g$ are positive real numbers, so
the rational function $f/g$ has no zeros or poles on $X_{> 0}.$
The integral of the $n$-form $\, (f/g) \Omega_X \,$
over the positive toric variety $X_{> 0}$ is a positive real number, or it diverges.
Our aim is to compute this number numerically.
We focus on integrals of interest in Bayesian statistics.

The main contribution of this article is 
a geometric theory of Monte Carlo sampling,
based on the positive geometry $(X,X_{>0}).$
A central role is played by the
canonical form $\Omega_X.$
The approach was first introduced in
\cite{borinsky2020tropical} for
Feynman integrals on projective space $X= \PP^n.$

Our presentation is organized as follows.
\Cref{sec2} reviews
the quotient construction of a toric variety $X$
from its Cox ring. The canonical form $\Omega_X$ is defined in~\eqref{eq:canonicalform}.
We introduce integrals of the form $\int_{X_{>0}} (f/g) \Omega_X,$ and
we present a convergence criterion in \Cref{thm:convergent}.

In \Cref{sec3}, we replace the rational function $f/g$ in the integrand 
by its tropicalization. The resulting piecewise monomial structure
divides the positive toric variety $X_{>0}$ into sectors.
\Cref{thm:sectorformula}
gives a formula for integrating over each sector,
against the tropical probability distribution on $X_{>0}.$
\Cref{alg:eins} offers a method for sampling from that distribution. 
Although we here focus on its use in Bayesian statistics, we stress that the method of replacing densities by their tropical approximation is not 
custom-tailored for Bayesian purposes; the idea is applicable and might be beneficial more widely in statistics.

In \Cref{sec4}, we develop a tropical approximation scheme for 
the classical integral\linebreak $\,\int_{X_{>0}} (f/g) \Omega_X.$
We apply rejection sampling to draw from the
density induced by $f/g$ on~$X_{>0}$ with its canonical form $\Omega_X.$
The runtime is analyzed in terms of 
the acceptance rate. 

\Cref{sec5} is devoted to discrete statistical models
whose parameter space is a simple polytope~$P.$
Familiar instances are
linear models, toric models, and their mixtures \cite{Sullivant}.
We show how to pull back Bayesian priors from $P$ to $X_{>0}$ via the moment map.
The push-forward of $\Omega_X$ to $P$ gives rise to 
the Wachspress model whose states are the vertices of~$P.$
The section concludes with a combinatorial analysis of the coin model in
Equation \eqref{eq:runningex1}.

In \Cref{sec6}, we apply tropical integration and tropical sampling
to data analysis in the Bayesian setting. We focus on the
statistical models from \Cref{sec5}, but now
lifted from $P$ to $X_{>0}.$ We present algorithms, along with their
implementation, for computing marginal likelihood integrals. 
Sampling from the posterior distribution is also~discussed.
Our software and other supplementary material for this article is 
available at the repository website {\tt MathRepo}~\cite{mathrepo} of MPI MiS via the link
\href{https://mathrepo.mis.mpg.de/BayesianIntegrals/}{https://mathrepo.mis.mpg.de/BayesianIntegrals}$\,.$

\section{Toric Varieties and their Canonical Forms} \label{sec2}
We review the set-up of toric geometry,
leading up to the integrals  studied in this paper.
For complete details on toric varieties we refer to the textbook
\cite{cox2011toric} and to the  notes~\cite{telentoric}.
Let $T$ be an $n$-dimensional complex algebraic torus with character lattice $M$ and co-character lattice $N = \Hom_{\ZZ}(M,\ZZ).$ Fixing an isomorphism $T \simeq (\CC^*)^n$ corresponds to identifying $M \simeq \ZZ^n$ and $N \simeq \ZZ^n.$ 
We write $\chi^a$ for the character in $M$ corresponding to the lattice point $a \in \ZZ^n$ and $\lambda^v$ for the co-character in $N$ corresponding to $v \in \ZZ^n.$ The pairing $\langle \cdot, \cdot \rangle : N \times M \rightarrow \ZZ$ is given by $\langle \lambda^v, \chi^a \rangle = \chi^a \circ \lambda^v \in \Hom_{\ZZ}(\CC^*,\CC^*) \simeq \ZZ.$ 
In coordinates, this is the
dot product $\langle v, a \rangle = v \cdot a.$

Fix a complete fan $\Sigma$ in $\RR^n = N \otimes_{\ZZ} \RR.$ The $n$-dimensional toric variety
$X = X_\Sigma \supset T$ is normal and complete.
Write $\Sigma(d)$ for the set of cones of dimension $d$ in $\Sigma,$ and
$k = |\Sigma(1)|$ for the number of rays of $\Sigma.$
Each ray $\rho \in \Sigma(1)$ has a primitive ray generator $v_\rho \in \ZZ^n,$
satisfying $\rho \cap \ZZ^n = \NN \cdot \, v_\rho.$ 
We collect the rays in the columns of the $n \times k$ matrix $V = [v_1 \,v_2 ~ \cdots ~ v_k].$ This matrix has more columns than rows, i.e., $k > n$, since $\Sigma$ is complete. 

The free group of torus-invariant Weil divisors on $X$ is $\text{Div}_T(X) = \bigoplus_\rho \ZZ \cdot D_\rho \simeq \ZZ^k.$ A character $\chi^a \in M$ extends to a rational function on $X$ with divisor $\text{div}(\chi^a) = \sum_{\rho} \langle v_i, a \rangle D_\rho.$ The transpose matrix $V^\top,$ viewed as a map of lattices $M \rightarrow \text{Div}_T(X) \simeq \ZZ^k,$ sends a character to its divisor. Two torus-invariant divisors $D_1, D_2 \in \text{Div}_T(X)$ are linearly equivalent if and only if $D_1 - D_2 = \text{div}(\chi^a)$ for some character~$\chi^a.$ Equivalently, there is an exact sequence 
\begin{equation} \label{eq:SESCl}
	0 \, \longrightarrow \, M \overset{V^\top}{\longrightarrow} \ZZ^k \,\longrightarrow\, \text{Cl}(X) \,\longrightarrow \,0.
\end{equation}
The cokernel $\text{Cl}(X) = \ZZ^k/\mathrm{im} V^\top$ is the \emph{divisor class group} of $X.$ The \emph{Picard group} $\text{Pic}(X) \subseteq \text{Cl}(X)$ is the subgroup of Cartier divisors modulo linear equivalence.
Applying the functor $\Hom_{\ZZ}(-, \CC^*)$ to  \eqref{eq:SESCl}, we obtain
the following exact sequence of multiplicative abelian groups:
\begin{equation} \label{eq:dualSESCl}
	1\, \longrightarrow  \,G \, \longrightarrow \, (\CC^*)^k \, \longrightarrow\, (\CC^*)^n \, \longrightarrow \, 1. 
\end{equation}
The group $G = \Hom_{\ZZ}(\text{Cl}(X), \CC^*)$ is reductive: it is a quasi-torus of dimension $k-n.$

We will now recall the definition of the Cox ring of $X$. Consider the polynomial ring $S = \CC[x_1,\ldots,x_k],$
with one variable $x_\rho$ for each divisor $D_\rho,$ $\rho \in \Sigma(1).$
The sequence \eqref{eq:SESCl} defines a grading of $S$ by the group ${\rm Cl}(X),$
and the sequence \eqref{eq:dualSESCl} gives the associated quasi-torus action by $G$ on the
affine space ${\rm Spec}(S) = \CC^k.$ The grading is as follows: 
\[  S\,\,= \bigoplus_{\gamma \in \text{Cl}(X)}\! S_\gamma, \qquad \text{where} \qquad S_{\gamma} \,\,= \!\!
\bigoplus_{a\,:\,V^\top a + c \geq 0} \!\! \CC \cdot \, x^{V^\top a \,+\, c}. \]
The vector $c \in \ZZ^k$ is fixed.
It represents any divisor $D_c = \sum_{\rho} c_\rho D_\rho$ such that $[D_c] = \gamma.$
The sum on the right is over all $a \in M$ such that the integer vector $V^\top a + c $ is nonnegative.

The toric variety $X$ can be realized as a quotient $(\CC^k \backslash \mathcal{V}(B)) /\!/ G.$
The {\em irrelevant~ideal} $B \triangleleft S $ is generated by the squarefree monomials
$x^{\hat{\sigma}} = \prod_{\rho \notin \sigma} x_\rho$ representing maximal cones,~i.e.,
\[ B \,\,=\,\, \langle \,x^{\hat{\sigma}} ~|~ \sigma \in \Sigma(n) \,\rangle. \]
The map $(\CC^*)^k \rightarrow (\CC^*)^n$ in \eqref{eq:dualSESCl} is constant on $G$-orbits. 
It is the restriction of the map $\pi : \CC^k \backslash \mathcal{V}(B) \rightarrow X$ which presents $X$ as the quotient above.
The notation $/\!/$ indicates that this is generally not a \emph{geometric} quotient. However, it is if $\Sigma$ is a simplicial fan.
Under this extra assumption, we write
$X = (\CC^k \backslash \mathcal{V}(B)) / G$ and there is a one-to-one correspondence 
\[ \{ \text{$G$-orbits in $\CC^k \backslash \mathcal{V}(B)$} \} \,\,\stackrel{1:1}{\longleftrightarrow} \,\,\{ \text{points in $X$} \}. \]
The  polynomial ring $S$, with its ${\rm Cl}(X)$-grading and irrelevant ideal $B$, is the \emph{Cox ring} of $X.$ 

The zero locus in $\CC^k$ of a homogeneous polynomial $f \in S$ is stable under the $G$-action.
Hence, the zero locus of
$f$ in $X $ is well-defined. In fact, homogeneous ideals of $S$ define subschemes of~$X,$
and all subschemes arise in this way. If $X$ is smooth, then
the  subschemes of~$X$ 
are in one-to-one correspondence
with the $B$-saturated homogeneous ideals of $S.$ 

If $f, g \in S$ are homogeneous of the same degree and $g \neq 0$, the quotient $f/g$ gives a rational function on $X$, defined on the open subset $\pi \left ((\CC^*)^k \setminus({\cal V}(B) \cup {\cal V}(g)) \right)$ via $(f/g)(p) = f(x)/g(x)$ where $x$ is any point in the $G$-orbit $\pi^{-1}(p)$. In the rest of the paper, we will use homogeneous rational functions of degree $0$ in the fraction field of $S$ to denote the corresponding rational function on $X$. Similarly, meromorphic differential forms on $X$ can be represented by rational functions in Cox coordinates, as in \eqref{eq:keyinsight} for $X = \PP^1 \times \PP^1 \times \PP^1$. See also \Cref{rmk:canonical}.

The material above may look overly formal to a novice. Yet,
toric varieties $X$ and their Cox coordinates $x_1,\ldots,x_k$ are
practical tools for applications, e.g.,~in the numerical solution of
polynomial equations \cite{telennumerical}. The present paper extends the
utility of the abstract setting to numerical computing at the interface of
statistics and physics \cite{borinsky2020tropical, sturmfels2020likelihood}.
In applications, the toric variety $X$ is usually projective,
i.e.,~$\Sigma$ is the normal fan of a lattice polytope in $\RR^n.$

\begin{example}[3-cube]
	Let $n=3,$ $k=6$ and $\Sigma$ the fan given by the eight orthants in $\RR^3.$
	Then $X = \PP^1 \times \PP^1 \times \PP^1,$ 
	with Cox ring $S = \CC[x_0,x_1,s_0,s_1,t_0,t_1],$ 
	graded by ${\rm Cl}(X) = \ZZ^3.$ 
	The irrelevant ideal is
	$\,B = \langle x_0,x_1 \rangle \,\cap\, \langle s_0,s_1 \rangle \,\cap \, \langle t_0,t_1 \rangle \,=\, \langle x_0 s_0 t_0, x_0 s_0 t_1, \ldots, x_1 s_1 t_1 \rangle.$
	We represent $X$ as the quotient of $\CC^6 \backslash \mathcal{V}(B)$ modulo the action
	$G = (\CC^*)^3,$ as in the Introduction.
	See \cite[Example 6.2.7 (2)]{maclagan2009introduction}
	for a detailed study of this example and its tropicalization.
\end{example}

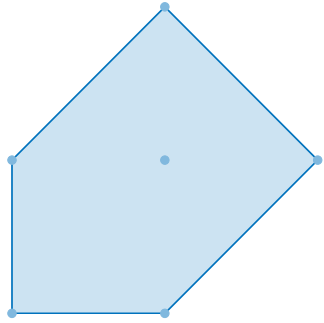
\begin{figure}[ht]
	\centering 
	\begin{tikzpicture}[scale=0.8] \begin{axis}[ width=2in, height=2in, scale only axis, xmin=-0.0, xmax=2.0, ymin=-0.0, ymax=2.0, ticks = none, ticks = none, axis background/.style={fill=white}, axis line style={draw=none} ] \addplot [color=mycolor1,solid,thick, fill = mycolor1!20!white,forget plot] table[row sep=crcr]{ 0 0\\
0 1\\
1 2\\	 2 1\\
1 0\\
0 0\\
}; \addplot[only marks,mark=*,mark size=2.0pt,mycolor1!50!white ] table[row sep=crcr]{ 0 0\\
1 0\\
0 1\\
1 1\\
2 1\\
1 2\\
}; \end{axis} \end{tikzpicture}
	\caption{This pentagon specifies the
		projective toric surface in \Cref{ex:pentagon1}.}
	\label{fig:pentagon}
\end{figure}

\begin{example}[Pentagon]  \label{ex:pentagon1}
	Fix $n=2,$ $k=5,$ and the polygon in Figure~\ref{fig:pentagon}.
	The rays of its normal fan $\Sigma$ are the inner normals to the edges. We write their generators in the matrix
	\begin{equation} \label{eq:Vmatrix}
		V \,\,=\,\, \begin{pmatrix}
			\,\,1 & \phantom{-}1 & -1 & -1 & \, 0 \,\,\\
			\,\,0 & -1 & -1 &  \phantom{-}1 & \, 1 \,\,
		\end{pmatrix} . 
	\end{equation}
	We have
	$G \simeq (\CC^*)^3,$ since
	${\rm Cl}(X) = \ZZ^5/{\rm im}\,V^\top $ is isomorphic to $\ZZ^3.$
	The irrelevant ideal is
	$$ B = \langle x_1 x_2 x_3, x_2x_3x_4, x_3x_4x_5, x_4x_5x_1, x_5x_1x_2 \rangle
	\,\,\triangleleft \,\, S = \CC[x_1,x_2,x_3,x_4,x_5]. $$
	The map 
	$(x_1, \ldots, x_5) \! \mapsto \! \!\!\; (x_1 x_2 x_3^{-1} x_4^{-1}\! , x_2^{-1} x_3^{-1}x_4x_5)$
	represents  $X$ as the quotient $(\CC^5 \setminus \mathcal{V}(B))/G.$
	We find it convenient to write the image of $V^\top$ in $\ZZ^5$ as the kernel of another matrix,~e.g.,
	\begin{equation}
		\label{eq:Wmatrix}
		W \,\,=\,\, \begin{pmatrix}
			0 & 1 & 0 & 1 & 0 \\
			1 & 0 & 1 & 0 & 1 \\
			2 & 0 & 1 & 1 & 0 \\
		\end{pmatrix}.
	\end{equation}
	The $\ZZ^3$-grading of $S$ sends $x_i$ to the 
	$i$-th column of $W.$ The $G$-action on $\CC^5$ is 
	$$ x_1 \,\mapsto \,t_2 t_3^2 \, x_1,\,\,
	x_2 \,\mapsto \,t_1 x_2,\,\,     x_3 \,\mapsto \,t_2 t_3  x_3,\,\,
	x_4 \,\mapsto \,t_1 t_3 x_4,\,\,     x_5 \,\mapsto \, t_2  x_5. $$
	If $c \in \ZZ^5$ then $\gamma = W c \in \ZZ^3$ represents the class $[D_c] \in {\rm Cl}(X)$
	of the divisor $D_c$ on~$X$.
\end{example}

The positive orthant $\RR^{k}_{>0} $ is disjoint from $\mathcal{V}(B)$ since
$B$ is a monomial ideal. We can restrict the quotient map
$\CC^k \backslash \mathcal{V}(B) \rightarrow X$ to $\RR^k_{>0}.$
The image of this restriction is the
{\em positive part} $X_{> 0}$ of the toric variety $X.$
The Euclidean closure of $X_{>0}$ in $X$ is denoted $X_{\geq 0}.$ 
If $X$ is projective with polytope $P$ then 
the moment map gives a homeomorphism from~$X_{> 0}$~onto~$P^\circ.$

Motivated by statistics (see Equation \eqref{eq:mli}), we wish to integrate \emph{meromorphic} $n$-forms with poles outside $X_{>0}$ over the nonnegative part $X_{\geq 0}.$ 
We here describe an explicit representation of such forms
via the Cox ring~$S.$ 
For any $n$-element subset $I \subset \Sigma(1),$ let $\det(V_I)$ denote the minor of $V$ indexed by $I.$ 
We define a meromorphic $n$-form on $\CC^k$ as follows:
\begin{equation} \label{eq:canonicalform}
	\Omega_X \,\,\,=\,\, \sum_{\substack{I \subset \Sigma(1),\, |I| = n}}  \det(V_I) \, \bigwedge_{\rho \in I} \frac{{\rm d} x_\rho}{x_\rho}.
\end{equation}
\vspace*{-1mm}

\noindent This is the \emph{canonical form} of the pair $(X,X_{\geq 0}),$ 
viewed as a \emph{positive geometry}, as explained by
Arkani-Hamed, Bai,  and Lam  in
\cite[Sect.~5.6.2]{arkani2017positive}. 
The canonical form $\Omega_X$
is the pullback under the quotient map $\pi$ of the $T$-invariant $n$-form $\bigwedge_{j = 1}^n \frac{{\rm d} t_j}{t_j} \in \Omega_T^n(T),$ where $t_j = \chi^{e_j}$ are coordinates on $T.$ This follows from \cite[Cor.~8.2.8]{cox2011toric}
by observing that $\pi^*(t_j) = \prod_{i=1}^k x_\rho^{\langle v_i, e_j \rangle}.$ By homogeneity, $\Omega_X$ can be viewed as a meromorphic form on $X.$
We will use $\Omega_X$ to denote this form on $\CC^k,$ on $X,$  as well as its restriction to~$X_{> 0}$ without mentioning the respective~transition. 

\smallskip

After scaling by a rational function,
the canonical form $\Omega_X$ defines a probability measure on $X_{> 0}.$
Given a rational function $f/g$ on $X$,
we are interested in the definite integral of the differential form $(f/g)\,\Omega_X$ over 
the positive toric variety $X_{>0}$. In symbols, this equals
\begin{equation}
	\label{eq:ourintegral}
	\mathcal{I} \,\,\,=\,\,\,
	\int_{X_{> 0}} \frac{f}{g} \,\Omega_X. 
\end{equation}

The integrals $\mathcal{I}$ 
appear as Feynman integrals in physics. The article
\cite{borinsky2020tropical} introduces tropical sampling
for Feynman integrals over the projective space $X = \mathbb{P}^n.$
The present paper generalizes that approach to the setting
where $X$ can be any toric variety.

\begin{remark} \label{rmk:canonical}
	The canonical form $\Omega_X$ of a toric variety $X$ is closely related 
	to the canonical sheaf $\omega_X.$ Indeed, by the discussion following 
	\cite[Corollary 8.2.8]{cox2011toric}, $\omega_X$ is  the sheaf of the cyclic graded $S$-module generated by the $n$-form $\left(\prod_{i=1}^k x_i \right ) \Omega_X.$ See also \cite[Proposition 2.1]{cox1996toric}. 
\end{remark}

We next explain how to understand and evaluate the integral \eqref{eq:ourintegral}.
Let $\gamma = \deg(f) = \deg(g) \in {\rm Cl}(X)$ be
the class of the divisor $D_c.$ The canonical isomorphism 
\[ S_{\gamma} \,\,= \!\!
\bigoplus_{a\,:\,V^\top a + c \, \geq 0} \!\! \CC \cdot \, x^{V^\top a \,+\, c} \, \, \simeq \bigoplus_{a\,:\,V^\top a + c \, \geq 0} \CC \cdot \, t^{a} \]
\vspace*{-2mm}

\noindent represents \emph{dehomogenization}. It sends $f$ and $g$ to Laurent polynomials $\hat{f}$ and $\hat{g},$ respectively. 

This is compatible with the quotient map $\pi: \CC^k \backslash {\cal V}(B) \rightarrow X$ in the following~way. The map of tori $\pi_{|(\CC^*)^k}: (\CC^*)^k \rightarrow T \subset X$ realizes the torus of $X$ as a geometric quotient \mbox{$T \simeq (\CC^*)^k / G.$} Further restricting to the positive part gives \mbox{$T_{>0} \simeq X_{>0}.$} Let \mbox{$\phi: X_{>0} \rightarrow T_{>0}$} denote this diffeomorphism. In Cox coordinates, $\phi$ is the monomial map given by the rows of $V,$ see \Cref{ex:pentagon1}. One checks easily that the functions $(f/g): X_{>0} \rightarrow \CC$ and \mbox{$(\hat{f}/\hat{g}): T_{>0} \rightarrow \CC$} satisfy $(f/g) = (\hat{f}/\hat{g}) \circ \phi.$ Moreover, $\Omega_X $ restricted to its dense torus is the form $\pi^* \left (\bigwedge_{j = 1}^n \frac{{\rm d} t_j}{t_j}\right ) $ which,  in turn, uniquely determines $\Omega_X.$
Those observations~imply the following proposition.

\begin{proposition} \label{prop:familiar}
	The integral  $\mathcal{I}$ in \eqref{eq:ourintegral} equals
	a more familiar integral over the positive orthant $T_{>0} = \RR^n_{>0},$ namely
	\begin{equation} \label{eq:familiarintegral}
		\mathcal{I}
		\,\,\,= \,\,\, \int_{T_{ > 0}} \frac{\hat{f}}{\hat{g}} \, \bigwedge_{j = 1}^n \frac{{\rm d} t_j}{t_j}.
	\end{equation}
\end{proposition}

We conclude Section \ref{sec2} with a result on convergence.
This \ref{thm:convergent} generalizes \cite[Theorem~3]{borinsky2020tropical}.

\begin{theorem}\label{thm:convergent}
	Suppose that the Newton polytope of the denominator  $g$ is $n$-dimensional
	and contains that of the numerator $f$ in its
	relative interior.
	Then the integral \eqref{eq:ourintegral} converges.
\end{theorem}

\begin{proof}
	We use the formulation in \eqref{eq:familiarintegral}.
	By linearity of the integral, it suffices to consider the case when
	$f$ is a monomial. In that  special case, our integral can be
	viewed as the Mellin transform of
	the polynomial $g$. Hence the analysis by Nilsson and Passare in~\cite{NilssonPassare} 
	applies here.
	Their result is stated for integrals over  $T_{\geq 0} = \RR^n_{\geq 0},$ so we can use it for~\eqref{eq:familiarintegral}. 
	The convergence result then follows from \cite[Theorem 1]{NilssonPassare}.
\end{proof}

\begin{remark} 
The hypothesis of \Cref{thm:convergent} will be satisfied for the Bayesian 
integrals that arise from our statistical models in Sections \ref{sec5} 
and \ref{sec6}. An example is the integrand in~\eqref{eq:unif}. The Newton 
polytope of the denominator is the standard $3$-cube scaled by a factor of 
two. Its unique interior lattice point is the Newton polytope of the monomial in the 
numerator.

In fact, it can be proven that the hypothesis of 
\Cref{thm:convergent} is also necessary for the convergence of 
\eqref{eq:ourintegral} as long as the polynomials $f$ and $g$ have positive 
coefficients. 
Hence, we can expect convergent statistical integrals over ratios of such polynomials to fulfill it.
\end{remark}

\section{Tropical Sampling} \label{sec3}
Our aim is to  evaluate the integral $\,\mathcal{I}\,$ in \eqref{eq:ourintegral} 
and \eqref{eq:familiarintegral}.
To this end,  we consider a tropicalized version of
the integral.  Following \cite{borinsky2020tropical}, we define the \emph{tropical approximation} of a polynomial $f \in S = \mathbb{C}[x_1, \ldots, x_k]$ to be the piecewise monomial function
\[ f^{\rm tr}\colon\,  \mathbb{R}_{> 0}^k \longrightarrow \mathbb{R}_{> 0}, \quad x ~ \mapsto \max_{\ell \in \operatorname{supp}(f)}  x^\ell .  \]
This differs in two aspects from the textbook definition of tropicalization
in \cite{maclagan2009introduction}. First, we adopt the max-convention.
Second, we use monomials $x^\ell$ instead of linear forms $\langle x,\ell \rangle.$
Thus $f^{\rm tr}$ is the exponential of the piecewise-linear convex function ${\rm trop}(f)$
usually derived from $f$.

If $f \in S \backslash \{0\}$ is homogeneous with positive coefficients, then
the ratio $f(x)/f^{\rm tr}(x)$ is a well-defined function on $\mathbb{R}^k_{>0} \subset \mathbb{C}^k \backslash \mathcal{V}(B)$ 
and  is constant on $G$-orbits. It induces a function $X_{>0} \rightarrow \mathbb{R}_{>0},$ which takes $x \in X_{>0}$ to $f(x')/f^{\rm tr}(x')$ for any $x' \in \pi^{-1}(x).$ 
Employing a slight abuse 
of notation,  the ratio $f(x)/f^{\rm tr}(x)$ also
denotes that function on $X_{>0}.$

As a special case of \cite[Theorems 8A and 8B]{borinsky2020tropical}, where also polynomials with negative or complex 
coefficients are allowed, such functions are bounded above and below:

\begin{proposition}\label{prop:bounded}
	Suppose that $f(x) = \sum_{\ell \in \operatorname{supp}(f)} f_\ell \, x^\ell$ has positive coefficients, 
	and set $\,C_1 = \min_{\ell \in \operatorname{supp}(f)} f_\ell\,$ and 
	$\,C_2 = \sum_{\ell \in \operatorname{supp}(f)} f_\ell.$
	Then, we have
	\[
	0 \,<\, C_1 \,\leq \,\frac{f(x)}{f^{\rm tr}(x)} \,\leq\, C_2 \,< \,\infty \quad \text{for all $\,x \in X_{>0}$}.
	\]
\end{proposition}

We assume from now on that
$f$ and $g$ are homogeneous polynomials in $S,$ with positive coefficients,
of the same degree in ${\rm Cl}(X),$ and that the hypothesis of Theorem 
\ref{thm:convergent} is satisfied.

\begin{corollary}
	The following integral over the positive toric variety is~finite:
	\begin{equation}
		\label{eq:Itr}
		\mathcal I^{\rm tr}
		\,\,\,=\,\,\int_{X_{> 0}}
		\frac{f^{\rm tr}}{ g^{\rm tr} } \,\Omega_X.
	\end{equation}
\end{corollary}

\begin{proof}
	The tropical function $f^{\rm tr}/g^{\rm tr}$ is positive on $X_{> 0}$
	and it is bounded above
	by a constant times the classical function $f/g.$ This follows by
	applying Proposition  \ref{prop:bounded} to both $f$ and $g.$
	Since the integral over $f/g$ is finite, so is the integral over~$f^{\rm tr}/g^{\rm tr}.$
\end{proof}

The function $f^{\rm tr}/g^{\rm tr}$ is piecewise monomial on $X_{> 0}.$
The pieces are the sectors to be described below. 
The integral of each monomial over its sector is given in Theorem 
\ref{thm:sectorformula}.
The value of
$\mathcal I^{\rm tr}$ is the sum
\eqref{eq:integralsum}
of  the sector integrals.
Here is an illustration:

\begin{example}[Classical integral versus tropical integral] \label{ex:trivial}
	We fix the projective line $X = \PP^1$ with
	coordinates
	$(x_0:x_1).$ The  following binary cubics satisfy our convergence hypotheses:
	$$ f \,=\, x_0^2 x_1 \quad {\rm and} \quad g \,=\, (x_0 + x_1)(x_0+3 x_1)(5 x_0 + x_1) . $$
	The corresponding tropical polynomial functions on the line segment $X_{\geq 0}=\PP^1_{\geq 0}$ are
	$$ f^{\rm tr} \,=\, f = x_0^2 x_1 \quad {\rm and} \quad 
	g^{\rm tr} \,=\, \begin{cases} \,x_0^3 & {\rm if} \,\,x_0 \geq x_1, \\ \,x_1^3 & {\rm if} \,\,x_0 \leq x_1. \end{cases} $$
	The classical integral \eqref{eq:ourintegral} equals
	$\frac{1}{56} (6 \,{\rm ln}(3)- {\rm ln}(5)) = 0.088968....$
	We find this on either chart $\{x_0=1\}$ or $\{x_1=1\}.$
	The tropical integral \eqref{eq:Itr} evaluates to $1 + \frac{1}{2} = \frac{3}{2}.$
	We integrate the two monomials in $f^{\rm tr}/g^{\rm tr}$ over
	the sectors $\{ x_0 \geq x_1\}$ and $\{x_0 \leq x_1\}.$
\end{example}

Since the tropical integral is easier to compute, we now rewrite~\eqref{eq:ourintegral} as
\begin{equation} \label{eq:omegatr}
	\mathcal{I} \,\,=\,\,
	\int_{X_{ > 0}}\! \frac{f}{ g } \,\Omega_X
	\,\,\,
	=\,\,\,
	\mathcal I^{\rm tr}
	\cdot
	\int_{X_{ > 0}} \!\!
	h
	\,\,
	\mu^{\rm tr}_{f,g}\, ,
\end{equation} 
where
\[ 
h \,=\, 
\frac{f \cdot g^{\rm tr}}{ g \cdot f ^{\rm tr} } 
\quad
\text{ and }
\quad 
\mu^{\rm tr}_{f,g}
\,=\,
\frac{1}{
	\mathcal I^{\rm tr}
}
\frac{f^{\rm tr}}{ g^{\rm tr} } \,\Omega_X. \qquad \qquad
\]
The function $h$ is positive and bounded on $X_{> 0},$ again by \Cref{prop:bounded}. 
The differential form $\mu^{\rm tr}_{f,g}$ is nonnegative on $X_{>0},$ and 
it integrates to $1.$ In symbols, 
$$\int_{X_>0} \mu_{f,g}^\text{tr}\,\,= \,\,1.$$

Viewed statistically, the following function 
is a {\em probability density} on $X_{> 0}$:
\begin{equation}\label{eq:dfgtr}
	d_{f,g}^{\,\text{tr}}\,\,\coloneqq \,\,\frac{1}{\mathcal{I}^\text{tr}} \frac{f^\text{tr}}{g^\text{tr}}.
\end{equation}
This density is given in terms of $\mathcal{I}^{\text{tr}}$ and the tropical approximations of $f$ and $g.$
For brevity, we refer to $d_{f,g}^{\,\text{tr}}$ as the {\em tropical density}. 
In those terms, $\mu_{f,g}^\text{tr} =d_{f,g}^\text{tr}\cdot \Omega_X$  is a {\em probability measure} 
on $X_{>0}$ and the pair $(X_{>0},\mu_{f,g}^\text{tr})$ is a {\em probability space}.  
For basic terminology from probability, we refer to the textbook~\cite{Stirzaker},
and to the guided tours in  \cite[Chapter 1]{Gelman} and \cite[Chapter 2]{Sullivant}.

If we can draw samples from the distribution defined by the tropical density, 
then we can use \emph{Monte Carlo integration} to estimate the integral~\eqref{eq:ourintegral}.
Furthermore, using \emph{rejection sampling}, we can also produce samples from the  classical  
density $d_{f,g} = \frac{1}{\mathcal I} \,\frac{f}{g}$ on~$X_{>0},$ where the value of the integral~\eqref{eq:ourintegral} plays the role of a 
normalization factor.
We will describe these computations in the next section.
They play a fundamental role in Bayesian inference.
For an introduction to Bayesian statistics~see~\cite{Gelman}.

In the remainder of this section, we present our \emph{tropical sampling algorithm},
for sampling from the probability distribution on~$X_{> 0}$ that is given by the tropical density
$d_{f,g}^{\,\text{tr}}.$
This algorithm was introduced in \cite{borinsky2020tropical} for the
special case of projective space $X = \PP^n,$
and it was successfully applied to Feynman integrals.
We here extend it to other toric varieties $X.$

The Newton polytopes $\mathcal{N}(f)$ and $\mathcal{N}(g)$ of 
the homogeneous polynomials $f$ and $g$ 
live in $\RR^k,$ but their dimension is at most $n,$ since
they lie in an affine translate  of ${\rm im}_\RR(V^\top) \simeq \RR^n.$
	In light of Theorem \ref{thm:convergent}, we assume that $\mathcal{N}(g)$ has the maximal dimension $n.$ 
We are interested in the normal fan of the
$n$-dimensional polytope $\mathcal{N}(f) + \mathcal{N}(g) = \mathcal{N}(fg),$ which lies in a different affine translate of ${\rm im}_\RR(V^\top).$
Its normal fan has the lineality space $K := {\rm ker}(V) \simeq \RR^{k-n},$ so
that fan can be seen as a pointed fan in $\RR^k / \, K \simeq \RR^n.$ 
We fix a simplicial refinement $\mathcal{F}$ of this normal fan.
Each maximal cone of $\mathcal{F}$ is spanned by $n$ linearly independent vectors,
and the union of these cones covers $\RR^k / \, K .$
We alert the reader that there are now two different fans:
$\Sigma$ is the fan of the toric variety~$X,$ whereas 
$\mathcal{F}$ comes from our polynomials $f$ and $g.$

\begin{example}
	\label{ex:genperm}
	In the application to Feynman integrals in \cite{borinsky2020tropical},
	the polynomials $f$ and $g$ are  {\em Symanzik polynomials}.
	Their Newton polytopes are {\em generalized permutohedra}
	\cite[Section~6]{borinsky2020tropical}. For such integrals,
	we can take $\mathcal{F}$ to be the fan determined by the
	hyperplanes $\{x_i = x_j\}.$
	The computational results in
	\cite[Section 7.4]{borinsky2020tropical} rely on this
	special combinatorial structure.
\end{example}

We now abbreviate $e^y = (e^{y_1}, \ldots, e^{y_k}),$ and we define the  \emph{exponential map} 
\begin{equation}
	\label{eq:exponentialmap}
	{\rm Exp}\,:\, \RR^k/\, K \,\rightarrow\, X_{>0}\,, \quad [ (y_1, \ldots, y_k)]\, \mapsto \, \pi(e^y) . 
\end{equation}
Here $\pi : \CC^k \setminus {\cal V}(B) \rightarrow X$ is the quotient map from Section \ref{sec2}. The map ${\rm Exp}$ is well-defined since the subspace $K$ is mapped into the image of  $G$ under the 
homomorphism
$\RR^k \rightarrow (\CC^*)^k,\, y \mapsto e^y,$ cf.~the exact sequences \eqref{eq:SESCl} and \eqref{eq:dualSESCl}. 

\begin{remark}
	The exponential map is an inverse to \emph{tropicalization}. The co\-ordi\-nate-wise logarithm  map 
	$\RR_{>0}^k \rightarrow \RR^k$  turns the multiplicative action of $G$ into an additive action of~$K.$ 
	That is, it induces a map ${\rm Log}: X_{>0} \rightarrow \RR^k/ \, K.$ We refer
	to \cite[Chapter~6]{maclagan2009introduction} for  details.
\end{remark}

We continue to retain the  hypotheses 
$\,{\rm dim}(\,\mathcal{N}(g)\,) = n\,$ and $\,\mathcal{N}(f) \subseteq {\rm relint}\,\mathcal{N}(g).$

\begin{lemma} \label{lem:geq0}
	For all nonzero elements \mbox{$y \in \RR^k / \, K,$} we have
	\begin{equation} \label{eq:geq0}
		\max_{\nu \in \mathcal{N}(g)} y \cdot \nu \,-\, \max_{\nu \in \mathcal{N}(f)} y \cdot \nu \,\,>\,\, 0.
	\end{equation} 
\end{lemma}

\begin{proof}
	The left hand side of the inequality is well-defined modulo $K,$ because
	both Newton polytopes lie in the same affine translate of $K.$ Suppose the
	two maxima are attained for $\nu_f \in \mathcal{N}(f)$ and $\nu_g \in \mathcal{N}(g).$
	If $y \cdot \nu_g \leq y \cdot \nu_f$ were to hold, then $\nu_f $ lies in both
	$\mathcal{N}(f)$ and the boundary of $\mathcal{N}(g).$
	This contradicts our hypothesis. We therefore have $y \cdot \nu_g > y \cdot \nu_f.$
\end{proof}

\begin{lemma} \label{lem:geq1}
	Fix a cone $\sigma $ in the simplicial fan $ \mathcal{F}$ and 
	consider any vertices $\nu_f$ and $\nu_g$ of the corresponding faces of 
	the Newton polytopes $\mathcal{N}(f)$ and $\mathcal{N}(g).$ Then
	\[ \qquad \frac{f^{\rm tr}(x)}{g^{\rm tr}(x)}
	\,\,=\,\, x^{-(\nu_g - \nu_f)}  \quad \text{for all $\,x \in \RR^k$ such that
		$\,\pi(x) \in {\rm Exp}(\sigma).$}\]
\end{lemma}

\begin{proof}
	By definition of the normal fan,  the two functions
	$y \mapsto \max_{\nu \in \mathcal{N}(f)} y \cdot \nu$ and 
	\mbox{$y \mapsto \max_{\nu \in \mathcal{N}(g)} y \cdot \nu$} are linear on the cone $\sigma.$
	Let $y \in \sigma $ satisfy $\operatorname{Exp}(y) = x.$
	We~have
	\[
	\log f^{\rm tr}(x) \,=\, \log \max_{\ell \in \textrm{supp}(f)} x^\ell \, = \,
	\max_{\ell \in \textrm{supp}(f)} \ell \cdot y
	\, =\,
	\max_{\nu \in \mathcal{N}(f)} \nu \cdot y \, = \, \nu_f \cdot y.
	\]
	Similarly, $ \log g^{\rm tr}(x) \,= \,\nu_g \cdot y.$
	Applying the exponential function yields the assertion.
\end{proof}

To evaluate the integral \eqref{eq:ourintegral},
we use the factorization in \eqref{eq:omegatr}. The first task is
to evaluate the tropical integral $\mathcal{I}^{\rm tr}.$
Since ${\rm Exp}$ is a bijection, we  use the decomposition
\begin{equation}
	\label{eq:integralsum} \mathcal{I}^{\rm tr} \, \, = \,\,
	\sum_{\sigma \in \mathcal{F}(n)} \! \mathcal{I}^{\rm tr}_\sigma
	\qquad {\rm where} \qquad
	\mathcal I^{\rm tr}_\sigma
	\,\,\,=\,\,\int_{{\rm Exp}(\sigma)}
	\frac{f^{\rm tr}}{ g^{\rm tr} } \,\Omega_X.
\end{equation}
The positive toric variety $X_{> 0}$ is partitioned into the
{\em sectors} ${\rm Exp}(\sigma).$ Each tropical integral $\,\mathcal{I}^{\rm tr}_\sigma$
is the integral of a  Laurent monomial  of degree zero, namely
$x^{-\delta_\sigma},$ where  $\delta_\sigma = \nu_g - \nu_f.$ 
The integral of $x^{-\delta_\sigma} $ over all of $X_{>0}$ diverges, but our set-up 
ensures that it converges on the sector ${\rm Exp}(\sigma).$
We saw this for $X = \PP^1$ in \Cref{ex:trivial}.

We next present a formula for  the integral
$\,\mathcal{I}^{\rm tr}_\sigma.$ We fix a matrix  $W = [w_1 \,\cdots \, w_n] \in \RR^{k\times n}$ 
whose $n$ column vectors $w_\ell$ generate the  simplicial cone $\sigma$ in $\RR^k / K .$

\begin{theorem} \label{thm:sectorformula}
	The tropical sector integral in \eqref{eq:integralsum} is equal to
	\begin{equation}
		\label{eq:detVU} \mathcal I^{\rm tr}_\sigma \,\,=\,\,
		\frac{\det(V W)}{\prod_{\ell = 1}^n w_\ell \cdot \delta_\sigma }. 
	\end{equation}
\end{theorem}

Before proving \eqref{eq:detVU}, we comment on its interpretation.
The columns $w_\ell$  of $W$ are only defined up to the equivalence in $\RR^{k}/K.$
Still, the numerator is well-defined, because $\det(VW) = \det(VW')$ for
any choices of representatives in the matrix $W' = \left[w_1' \, \cdots \, w_n'\right]$ with 
$w_i' = w_i + \lambda_i$ and $\lambda_i \in K=\ker V.$ 
As $x^{-\delta_\sigma}$ has degree 0, we have $\delta_\sigma = \nu_g - \nu_f \in {\rm im}_\RR(V^\top) = K^\perp.$ Thus, the denominator does not depend on the choice of representative for $w_\ell.$ 
Similarly, the lengths of the generators $w_\ell$ are irrelevant for the description of the cone. The quotient in \eqref{eq:detVU} is invariant under rescalings of a vector $w_\ell \rightarrow \xi w_\ell,$ as the multiplier $\xi,$ which factors out of the determinant, cancels between the numerator and the denominator.
The sign of the numerator depends on the ordering of the vectors $w_\ell$ in $W$ which is arbitrary a priori. 
We arrange them in the matrix $W$ so that the condition $\det(VW) > 0$ holds.
Hence, the value of $\mathcal I^{\rm tr}_\sigma$ is
an invariant of the cone $\sigma$ equipped with a positive orientation and the data $V$ and $\delta_\sigma.$
By \Cref{lem:geq0}, we have $w_\ell \cdot \delta_\sigma > 0$ for all $\ell$ and 
hence $0<\mathcal I^{\rm tr}_\sigma < \infty $ for all $\sigma.$

\begin{proof} From  \Cref{lem:geq1} and our discussion above, we know that
	$$ \mathcal{I}_\sigma^{\rm tr} \,\,= \, \int_{{\rm Exp}(\sigma)}\!\! \!\!\!\! x^{-\delta_\sigma} \,\Omega_X. $$
	We expand $\Omega_X$ as in \eqref{eq:canonicalform}, and we change coordinates 
	under the exponential map:
		\begin{equation} \label{eq:thisgives}
		\mathcal{I}_\sigma^{\rm tr} \,\,=\,\,
		\int_\sigma e^{-y \cdot \delta_\sigma} \! \sum_{\substack{ I \subset \Sigma(1), \\ |I| = n}} \!
		\det(V_I) \, \bigwedge_{i \in I} {\rm d} y_i  \quad =\,\,
		\sum_{\substack{ I \subset \Sigma(1), \\ |I| = n}} \!\!
		\det(V_I) \int_\sigma e^{-y \cdot \delta_\sigma} \bigwedge_{i \in I} {\rm d} y_i. 
	\end{equation}
	On the right is  the integral
	of the exponential of a linear form over a simplicial cone.
	Since $\sigma$ is the image of $\RR^n_{> 0}$ under the linear map
	given by the  matrix $W_I,$ we obtain
	$$   \int_\sigma e^{-y \cdot \delta_\sigma} \bigwedge_{i \in I} {\rm d} y_i \quad = \quad
	\frac{ {\rm det}(W_I)}{ \prod_{\ell = 1}^n w_\ell \cdot \delta_\sigma }.
	$$
	The Cauchy--Binet formula 
	$\,\det (VW) = \sum_I \det(V_I) \det(W_I)\,$
	now proves \eqref{eq:detVU}.
\end{proof}

The next result generalizes \cite[Lemma~17]{borinsky2020tropical}.
For an arbitrary kernel $\psi,$ the sector integral is transformed to the standard cube.
Our proof technique is adapted from~\cite{borinsky2020tropical}.

\begin{proposition} \label{prop:sector}
	For any bounded function $\psi: X_{> 0} \rightarrow \RR,$ we have
	\begin{equation} \label{eq:cubetrafo}
		\int_{{\rm Exp}(\sigma)} 
		\!
		\frac{f^{\rm tr}}{g^{\rm tr}} \,
		\psi \, \,\Omega_X \quad = \quad
		\mathcal I^{\rm tr}_\sigma \,\cdot \,
		\int_{[0,1]^n} \!\! \psi(x^\sigma(q)) \, {\rm d} q_1 \wedge  \cdots \wedge  {\rm d} q_n,
	\end{equation}
	where
	\begin{equation} \label{eq:xsigmaq}
		x_i^\sigma (q) \,\,= \,\,\prod_{\ell = 1}^n q_\ell^{-(w_{\ell})_i /(w_\ell \cdot \delta_\sigma)} \quad
		{\rm for} \,\,\, i = 1,2,\ldots,k. 
	\end{equation}
	In particular, the integral in \eqref{eq:cubetrafo} is finite for any cone $\sigma \in \mathcal{F}(n).$
\end{proposition}

\begin{proof}
	With the exponential map as in \eqref{eq:thisgives}, the integral on the left equals
	\begin{equation}
		\label{eq:phiintegral}
		\sum_{\substack{ I \subset \Sigma(1), \\ |I| = n}} \det(V_I) \cdot
		\int_\sigma e^{-y \cdot \delta_\sigma} \psi(
		{\rm Exp}(y)) \,  \bigwedge_{\rho \in I} {\rm d} y_\rho. \qquad
	\end{equation}
	Using coordinates $\lambda$ on $\RR^n_{>0},$ and writing
	$y_i = \sum_{\ell = 1}^n \lambda_\ell (w_{\ell})_i$,
	the integral in \eqref{eq:phiintegral} is
	$$ \qquad \qquad
	\det(W_I) \cdot \int_{\RR^n_{>0}} e^{- (W \lambda) \cdot \delta_\sigma}
	\psi({\rm Exp}(W \lambda)) \, \prod_{\ell = 1}^n \textrm{d} \lambda_\ell. $$
	We now transform the integral from $\RR^n_{>0}$
	to the cube $[0,1]^n$ using the logarithm~function. Namely, we 
	apply the transformation
	$\lambda_\ell = (w_\ell \cdot \delta_\sigma)^{-1} \log (q_\ell)$ for $\ell=1,\ldots,n.$
	This yields the right hand side in \eqref{eq:cubetrafo}.
	This last step uses the fact that $w_\ell \cdot \delta_\sigma > 0,$  which is
	known from \Cref{lem:geq0}.
	Together with boundedness of~$\psi,$ this implies convergence.
\end{proof}

\begin{remark}
	The formula \eqref{eq:xsigmaq} is a parameterization of each sector by a standard cube:
	$$ x^\sigma : [0,1]^n \rightarrow {\rm Exp}(\sigma) , \quad q \mapsto x^\sigma(q) .$$
	To digest the exponent in \eqref{eq:xsigmaq}, 
	note that the $(i,\ell)$-th entry of $W$ is the $i$th entry~$(w_{\ell})_i$ of the vector~$w_{\ell},$ that
	$\delta_\sigma$ lies in $K,$ and that the row vector 
	$\delta_\sigma\cdot W$  has coordinates~$w_\ell \cdot \delta_\sigma.$
\end{remark}

We use the sector decomposition of $X_{> 0}$
given by the fan $\mathcal{F}$ to evaluate the~integral \eqref{eq:ourintegral}.
Rewriting~\eqref{eq:ourintegral} as in~\eqref{eq:omegatr}, the parameterization $\,\operatorname{Exp}: \RR^{k}/K \rightarrow X_{> 0}\,$
gives
\begin{equation}
	\label{eq:secdecomp}
	\int_{X_{>0}} \frac{f}{g} \, \Omega_X 
	\,\,= \,\,
	\int_{X_{>0}} \frac{f^{\rm tr}}{g^{\rm tr} } \, \, h \, \Omega_X 
	\,\,= \sum_{\sigma \in  \mathcal{F} (n)} \int_{{\rm Exp}(\sigma)}\frac{f^{\rm tr}}{g^{\rm tr} } \, h \, \Omega_X. 
\end{equation}
Each integral on the right hand side is an integral over the cube by \Cref{prop:sector}.
These integrals over $[0,1]^n$ are suitable to be evaluated using black-box integration algorithms.
We implemented this in \texttt{Julia}, using the package \texttt{Polymake.jl} (v0.6.1) 
for polyhedral computations; see \cite{gawrilow2000polymake,kaluba2020polymake} and \cite{mathrepo}.
Moreover, specializing to $h=1$ in this representation of~\eqref{eq:ourintegral}
gives a method for computing the normalization factor 
$\,\mathcal I^{\rm tr} = \sum_{\sigma \in \mathcal{F}(n)} \mathcal I^{\rm tr}_\sigma \,$ 
in Equation~\eqref{eq:omegatr}.

\begin{remark} 
	The sector decomposition \eqref{eq:secdecomp} gives an alternative
	proof of \Cref{thm:convergent}, with no reference to \cite{NilssonPassare}.
	It would be interesting to undertake a more detailed study of the Mellin transform
from a tropical perspective.
\end{remark}

\Cref{prop:sector} also gives the desired algorithm to sample from the 
 distribution in~\eqref{eq:omegatr}.
As input we need the simplicial fan $\mathcal F,$
where each maximal cone $\sigma$ comes with the following data:
a generating set, 
the vector $\delta_\sigma$ 
that encodes the 
function $f^{\rm tr}/g^{\rm tr}$ as in \Cref{lem:geq1},
and the numbers $\mathcal{I}_\sigma^{\rm tr}$ in \eqref{eq:detVU}.
Hence we also know
$\mathcal{I}^{\rm tr}= \sum_{\sigma \in \mathcal F(n)} \mathcal{I}_\sigma^{\rm tr}.$

\begin{algo}[Sampling from the tropical density $d_{f,g}^{\,\text{tr}}$]
	\mbox{} \label{alg:eins} \smallskip \\
	\noindent \textbf{Input:} 
	$\mathcal{F}(n),\,\delta_\sigma, \, \mathcal{I}_\sigma^{\rm tr}$
	and $\,\mathcal{I}^{\rm tr}.$
	\begin{enumerate}
		\item Draw an $n$-dimensional cone $\sigma$ from $\mathcal{F}(n)$ with probability $\,\mathcal I_{\sigma}^{\rm tr}/\mathcal I^{\rm tr}.$
		\item Draw a sample $q$ from the unit hypercube $[0,1]^n$ using the uniform distribution.
		\item Compute $x^\sigma(q) \in \operatorname{Exp}(\sigma)$ with $x^\sigma(q)$ as in \Cref{prop:sector}.
	\end{enumerate}
	\noindent \textbf{Output:} The element $x^\sigma(q)\in X_{> 0} ,$ a sample from the probability space  $(X_{>0},\mu^{\rm tr}_{f,g}).$
\end{algo}

The vector $\delta_\sigma$ and the generators of $\sigma$ enter in the definition of the function $x^\sigma(q)$  in step~3.
To show the correctness of \Cref{alg:eins}, 
consider any bounded test function $\psi : X_{> 0} \rightarrow \RR.$
By \Cref{prop:sector}, the expected value of the function $(\sigma, q) \mapsto \psi(x^\sigma(q))$ on ${\cal F}(n) \times [0,1]^n,$
where $\sigma$ and $q$ are  sampled by steps 1 and 2, is
\[
\sum_{\sigma \in \mathcal{F}(n)}
\frac{\mathcal I^{\rm tr}_\sigma}{\mathcal I^{\rm tr}} \,
\int_{[0,1]^n} \! \psi(x^\sigma (q)) \, {\rm d} q_1 \wedge  \cdots \wedge  {\rm d} q_n
\\
\,\, =\,\,
\frac{1}{\mathcal I^{\rm tr}} \,
\sum_{\sigma \in \mathcal{F}(n)}
\int_{{\rm Exp}(\sigma)} 
\frac{f^{\rm tr}(x)}{g^{\rm tr}(x)} \,
\psi(x) \, \,\Omega_X.
\]
We conclude from \eqref{eq:omegatr}
that the sum on the right is equal to the desired expectation 
$$ \mathbb{E}_{\mu_{f,g}^{\text{tr}}}[\psi] \,\,\,= \,\,\,
\int_{X_{> 0}} \!\psi(x) \, \mu^{\rm tr}_{f,g}  \,\,=\,\, \int_{X_{>0}} \psi(x) d_{f,g}^{\,\text{tr}} \Omega_X \,\,\,=\,\,\,
\frac{1}{\mathcal I^{\rm tr}} \,
\int_{X_{> 0}} \, \psi(x) \frac{f^{\rm tr}(x)}{g^{\rm tr}(x)} \,\Omega_X  .
$$

To run \Cref{alg:eins} efficiently, we assume that the simplicial refinement  $\mathcal{F}$ of the normal fan of 
${\cal N}(fg)$ has been precomputed offline.
That computation can be time-consuming.
In the application to statistics, cf.~\Cref{sec6},
this is done only once for any fixed model.
Step 1 in \Cref{alg:eins} requires to sample from a finite set $\mathcal{F}(n)$
with a given probability distribution. With some preprocessing  (cf.~\cite{walker}),
this task can be performed in a runtime which is independent of the cardinality of $\mathcal{F}(n).$
The runtime of the algorithm is therefore independent of the size of the fan $ \mathcal{F},$ and it depends only linearly on the dimension $n$ of $X.$

\section{Numerical Integration and Rejection Sampling} \label{sec4}
In the previous section, we computed the tropical integral $\,\mathcal{I}^{\rm tr}$
and we explained how to sample from the tropical density. 
This will now be utilized in a non-tropicalized context. 
The domain of integration is the positive toric variety $X_{> 0}.$ 
We denote by  $\mu_{f,g}$ the measure
\begin{equation}
	\label{eq:classicaldensity}
	\mu_{f,g} \,\,\coloneqq\,\, \frac{1}{\mathcal I} \, \frac{f}{g}\, \Omega_X\, ,
	\quad 
	\text{ where }
	\quad
	\mathcal I \,\,=\, \,
	\int_{X_{> 0}} \frac{f}{g}\, \Omega_X.
\end{equation}
We assume that  $f$ and $g$ satisfy the convergence criteria from \Cref{thm:convergent}. Similarly to that in Equation \eqref{eq:dfgtr}, the following function is a probability density on $X_{>0}$:
\begin{equation} \label{eq:dfg}
	d _{f,g} \,\,\coloneqq \,\,\frac{1}{\mathcal{I}} \cdot \frac{f}{g}.
\end{equation}
Note that  $\mu_{f,g} = d_{f,g}\cdot \Omega_X.$
We regard $(X_{> 0}, \mu_{f,g})$ as the classical version of the tropical probability space
$(X_{> 0}, \mu_{f,g}^{\rm tr})$ which was introduced in Equation~\eqref{eq:omegatr}.
The normalizing constant $\mathcal{I}$ in \eqref{eq:classicaldensity} is the classical integral 
we saw in Equations~\eqref{eq:ourintegral} and \eqref{eq:familiarintegral}. 
The classical and tropical probability measures $\mu_{f,g}$ and $\mu_{f,g}^\text{tr}$ are 
related to each other by the formula
\begin{align}\label{eq:cltrmeasure} \mu_{f,g} \,\,=\,\, \frac{\mathcal{I}^\text{tr}}{\mathcal{I}} \cdot h \cdot \mu_{f,g}^\text{tr}. \end{align}

This section contains two novel contributions.
We  present a tropical Monte Carlo method for numerically evaluating $\mathcal{I},$
and we develop an algorithm for sampling from the density $d_{f,g}$ in \eqref{eq:dfg}.
Applications to statistics appear in Sections \ref{sec5} and \ref{sec6}.

We shall evaluate $\mathcal{I}$ using the formula in~\eqref{eq:omegatr},
by computing the expected~value  
\begin{align}\label{eq:expval} \mathbb{E}_{\mu_{f,g}^\text{tr}}[h] \,\, = \,\, \int_{X_{>0}} h \,\mu_{f,g}^\text{tr} \end{align}
with respect to the tropical measure~$\mu^{\rm tr}_{f,g}$ on 
the positive toric variety $X_{> 0}.$

\begin{corollary} \label{cor:intapprox}
	Suppose that \Cref{alg:eins} is used to
	draw $N$ i.i.d.\ samples $x^{(1)}, \ldots, x^{(N)} $ from the space $X_{> 0}$ with its tropical density.
	Then our  integral~\eqref{eq:ourintegral} approximately equals
	\begin{equation}
		\label{eq:montecarlo}
		\mathcal{I}
		\,\,\,\approx\,\,\,
		{\cal I}_N \, \, = \, \, \frac{\mathcal I^{\rm tr}}{N} \,\,
		\sum_{i=1}^N h\left( x^{(i)}\right).
	\end{equation}
\end{corollary}

To assess the quality of this approximation, we first observe that 
\Cref{prop:bounded}  yields bounds
in terms of the coefficients $f_{\ell}$ and $g_{\ell}$ of the given polynomials.
We~have
\begin{equation} \label{eq:hbounds}
	M_1 \,\leq\,
	h(x) \,\leq\, M_2 \qquad \text{ for all } \, x \in X_{> 0},
\end{equation}
where
\begin{equation}
	\label{eq:M1M2}
	M_1 \,\,=\,\, \frac{\min_{\ell \in \operatorname{supp}(f)} f_\ell}{\sum_{\ell \in \operatorname{supp}(g)} g_\ell} \qquad \text{and} \qquad 
	M_2 \,\,=\,\, \frac{\sum_{\ell \in \operatorname{supp}(f)} f_\ell}{\min_{\ell \in \operatorname{supp}(g)} g_\ell}.
	\qquad
\end{equation}

\begin{proposition} \label{prop:expdivi}
	The standard deviation of the approximation \eqref{eq:montecarlo} satisfies
	\begin{equation}
		\label{eq:generalbound}
		\sqrt{{\mathbb{E}}[\, ( {\cal I} - {\cal I}_N )^2 \, ]} \,\, \,\leq \,\,\, {\cal I}^{\rm tr} \cdot \sqrt{\frac{M_2^2-M_1^2}{N}}. 
	\end{equation}
\end{proposition}

\begin{proof}
	The expected value of the random variable ${\cal I}_N$ from \eqref{eq:montecarlo} equals ${\cal I}.$ 
	By the linearity of the variance for independent random variables, we have
	\[ 
	{\mathbb E}[({\cal I} - {\cal I}_N)^2] \,\,=\,\,
	{\rm Var}[{\cal I}_N] \,\,=
	\,\, \left ( \frac{{\cal I}^{\, \rm tr}}{N} \right )^{\! 2} {\rm Var}[h(x^{(1)}) + \cdots + h(x^{(N)}) ] 
	\,\,=\,\,  \frac{({\cal I}^{\, \rm tr})^2}{N}  {\rm Var}[ h ].  \]
	Using the bounds in \eqref{eq:hbounds}, we find that
	$\, \operatorname{Var} \left[ h \right] = \mathbb{E} \left[ h^2 \right] - \mathbb{E} \left[ h \right]^2 \leq M_2^2 -M_1^2. $
\end{proof}

Proposition  \ref{prop:expdivi} ensures that
the method in Corollary \ref{cor:intapprox}
correctly computes a numerical approximation of the integral $\mathcal I.$
The variance stays bounded and does not depend on~$n = {\rm dim}(X).$ 
Another application of the tropical approach in \Cref{alg:eins} is 
drawing from the probability density $d_{f,g}$
via \emph{rejection sampling}.   
In the next paragraph, we briefly review the overall principle of rejection sampling.
For further reading we refer to \cite[Section 10.3]{Gelman}.

Let $d_1$ and $d_2$ be densities on the same space with respect to the same differential form (e.g.~$X_{>0}$ with $\Omega_X$).
Suppose it is hard to sample from~$d_1,$
but sampling from $d_2$ is easy,
and we know a constant $C\geq 1$ such that \mbox{$d_1(x)/d_2(x)\leq C$} for all $x$ in the domain.
Our aim is to sample from $d_1$ using samples from $d_2.$ 
For this, we draw a sample $x$ using the distribution $d_2$ and a sample $\xi$ 
from the interval $[0,C]$ with uniform distribution.
We accept $x$ if $\xi <d_1(x)/d_2(x).$ 
Otherwise, we reject~$x.$ 
The density of producing an accepted sample from this process is 
$d_2(x) \cdot d_1(x)/d_2(x).$
So, accepted samples follow the density $d_1.$

We now apply rejection sampling to our problem. This is done as follows. The two densities of interest are
$d_1 = d_{f,g}$ and $d_2 = d_{f,g}^{\,\text{tr}}.$ From \eqref{eq:cltrmeasure} and \eqref{eq:hbounds}, we obtain 
\begin{equation}\label{eq:densitiesbounds} 
	d_{f,g} \,\, \leq \,\, \frac{\mathcal{I}^{\, \rm tr}}{\mathcal{I}}\cdot M_2 \cdot d_{f,g}^{\,\text{tr}}.
\end{equation}
We thus choose $\,C = ({\cal I}^{\, \rm tr}/{\cal I})\cdot M_2\,$ as our constant for rejection sampling. 
This suggests that rejection sampling requires us to compute the integral ${\cal I}.$ 
However,
 if $\xi$ is sampled uniformly from $[0,C],$ then $\xi' = ({\cal I}/{\cal I}^{\, \rm tr}) \cdot \xi$ is sampled uniformly from $[0,M_2]$ and $\xi < d_{f,g}(x)/ d^{\, \rm tr}_{f,g}(x)$ is equivalent to $\xi' < h(x).$ This leads to the following algorithm. 

\begin{algo}[Sampling from the density $d_{f,g}$] \mbox{} 
	\label{alg:mufg} \smallskip \\
	\noindent \textbf{Input:} The input from Algorithm~\ref{alg:eins} and the constant $M_2.$
	\begin{enumerate}
		\item Draw a sample $x $ from $ X_{> 0}$ using the tropical density $d_{f,g}^{\,\text{tr}}.$
		\item Draw a sample $\xi $ from the interval $ [0,M_2]$ using the uniform distribution.
		\item If $\xi < h(x),$ output $x.$ Otherwise reject the sample and start again.
	\end{enumerate}
	\noindent \textbf{Output:} The element $x\in X_{> 0} ,$ a sample from the probability space  $(X_{>0},\mu_{f,g}).$
\end{algo}
To check the validity of this algorithm, consider a bounded test function $\psi: X_{> 0} \rightarrow \RR.$ 
The expected value of $\psi,$ using the samples produced by 
\Cref{alg:mufg}, is equal to
\[
\mathbb{E}[\psi]  \,\,=\,\,
\frac{1}{D} \int_{X_{> 0} \times [0, M_2]} 
\psi(x) \, H\left(  h(x)-\xi \right) \,
\mu^{\rm tr}_{f,g} \wedge {\rm d} \xi.
\]
Here $D$ is a normalization factor that ensures $\mathbb{E}[1] = 1,$
and $H$ denotes the {\em Heaviside function}, i.e.,~$H(t) = 0$ for $t \leq 0$ and $H(t) = 1$ for $t > 0.$
By Equation~\eqref{eq:hbounds}, we have $h(x) \leq M_2$ for all $x \in X_{> 0}.$ 
By evaluating the inner integral over $\xi$ first, we obtain
\[\mathbb{E}[\psi] \,\,=\,\,\frac{1}{D} \int_{X_{> 0}} 
\psi(x) \,  h(x)\,\mu^{\rm tr}_{f,g}\,\,=\,\, \frac{1}{D}
\int_{X_{> 0}} \psi(x) \,  \frac{f}{g} \, \Omega_X . \]
This shows that $D ={\cal I},$
and we conclude that the expected value ${\mathbb E}[\psi]$ equals $ {\mathbb E}_{\mu_{f,g}}[\psi].$

The expected runtime of Algorithm~\ref{alg:mufg} is equal to the runtime of 
Algorithm~\ref{alg:eins} divided by the \emph{acceptance rate}. The latter is the probability that a sample drawn from $d^{\,\rm tr}_{f,g}$ results in a valid sample for $d_{f,g}.$
The bounds on $h$ in~\eqref{eq:hbounds} give rise to a lower bound for this~probability.

\begin{proposition} \label{prop:acceptancerate}
	The acceptance rate in Algorithm~\ref{alg:mufg} is at least $M_1/M_2.$
\end{proposition}
\begin{proof}
	The probability for acceptance of a sample equals
	\[
	\frac{1}{M_2}
	\int_{X_{> 0} \times [0, M_2]}  H \left( h(x) -\xi \right) \,
	\mu^{\rm tr}_{f,g} \wedge {\rm d} \xi \,\,=\,\,
	\frac{1}{M_2}\cdot
	\int_{X_{> 0}}  h(x)\, \mu^{\rm tr}_{f,g}
	\,\,\geq  \,\, \frac{1}{M_2}\cdot M_1.
	\]
	Here we used the lower bound from~Equation~\eqref{eq:hbounds}.
\end{proof}

The practical significance of \Cref{prop:acceptancerate}
comes from the fact that the lower bound does not depend 
on the dimension $n$ of the sample space $X_{>0}.$ 
This guarantees that---even for high-dimensional problems---the acceptance rate
in Algorithm~\ref{alg:mufg} remains strictly positive.

A word of caution is in order. If $f$ and $g$ have many terms with coefficients of roughly the same magnitude, it is clear that $M_1$ and $M_2$ are very small and very large, respectively. Moreover, in the statistical setting, the coefficients of $f$ and $g$ depend on the data vector, which was
called $u$ in the Introduction. We warn the reader that, despite dimension independence of the bounds, the efficiency of our sampling and integration approach in this setting declines when the entries of $u$ get large. We will see this in our computations of Section \ref{sec6}.

We conclude  with an example that illustrates the
material seen in this section.

\begin{example}[Pentagon] \label{ex:pentagon2}
	Let $X$ be the toric surface in
	\Cref{ex:pentagon1}.
	We consider the integral $\mathcal{I}$ in \eqref{eq:classicaldensity}
	and the probability density $d_{f,g}$ in \eqref{eq:dfg} defined by 
	$$ \begin{matrix}
		f &= & 
		2 \,x_1^2 x_2^2 x_3^3 x_4 x_5^3 \,\,+ \,\,
		3 \,x_1^2 x_2 x_3^2 x_4^2 x_5^4 \,\,+ \,\,
		5 \,x_1 x_2^2 x_3^5 x_4 x_5^2,
		\smallskip \\
		g &= &
		7 \,x_1^3 x_2^3 x_3^2 x_5^3 \,+ \,
		11 \,x_1^3 x_2 x_4^2 x_5^5 \,+ \,
		13 \,x_1 x_3^3 x_4^3 x_5^4 \,+ \,
		17 \, x_2^2 x_3^7 x_4 x_5.
	\end{matrix}
	$$
	Both $f$ and $g$ are homogeneous of degree $\gamma = (3,8,8)$ in the grading given by \eqref{eq:Wmatrix}. 
	We remark that $\gamma \in {\rm Cl}(X) \setminus {\rm Pic}(X)$ does not come from a Cartier divisor.
	The $16$ monomials of that degree are the lattice points in a quadrilateral whose normal fan is refined by $\Sigma.$ This is shown in green in Figure \ref{fig:polygonsEx45}. We see that 
	the orange triangle $\mathcal{N}(f)$ is contained in the interior of the purple quadrilateral $\mathcal{N}(g).$ 
	Hence \Cref{thm:convergent} ensures that the integral $\mathcal{I}$ converges.
	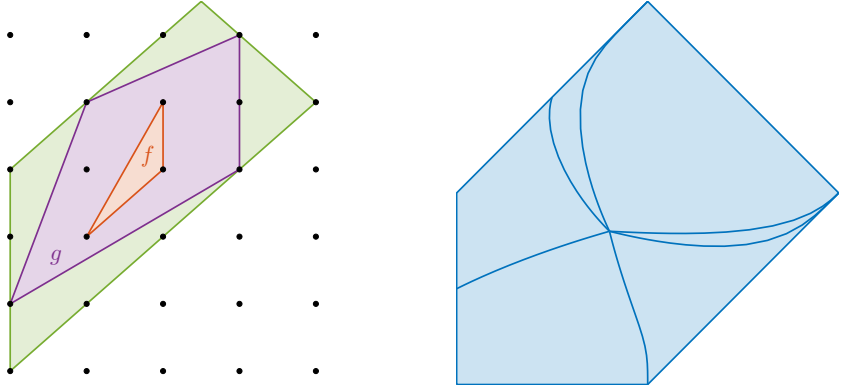
\begin{figure}[ht]
		\centering
		\begin{tikzpicture}[scale=0.8] \begin{axis}[ width=2.2in, height=2.5in, scale only axis, xmin=-2.2, xmax=2.2, ymin=-3.2, ymax=2.5, ticks = none, ticks = none, axis background/.style={fill=white}, axis line style={draw=none} ] \addplot [color=mycolor5,solid,thick, fill = mycolor5!20!white,forget plot] table[row sep=crcr]{ -2 -3\\
-2 0 \\
0.5 2.5\\
2 1\\
-2 -3\\
}; \addplot [color=mycolor4,solid,thick, fill = mycolor4!20!white,forget plot] table[row sep=crcr]{ -2 -2 \\
1 0\\
1 2\\
-1 1\\
-2 -2\\
}; \addplot [color=mycolor2,solid,thick, fill = mycolor2!20!white,forget plot] table[row sep=crcr]{ -1 -1\\
0 0\\
0 1\\
-1 -1\\
}; \addplot[only marks,mark=*,mark size=1.2pt,black ] coordinates { (-3, -3) (-2,-3) (-1,-3) (0,-3) (1,-3) (2,-3) (3, -3) (-3, -2) (-2,-2) (-1,-2) (0,-2) (1,-2) (2,-2) (3, -2) (-3, -1) (-2,-1) (-1,-1) (0,-1) (1,-1) (2,-1) (3,-1) (-3, -0) (-2,-0) (-1,-0) (0,-0) (1,-0) (2,-0) (3, -0) (-3, 1) (-2,1) (-1,1) (0,1) (1,1) (2,1) (3, 1) (-3, 2) (-2,2) (-1,2) (0,2) (1,2) (2,2) (3, 2) (-3, 3) (-2,3) (-1,3) (0,3) (1,3) (2,3) (3, 3) }; \node (A1) at (axis cs:2, 2) {}; \node (A2) at (axis cs:0, 1) {}; \node (B1) at (axis cs:-2,-2) {}; \node (B2) at (axis cs:-1, 0) {}; \node (C1) at (axis cs:-3,3) {}; \node (C2) at (axis cs:-2,2) {}; \node (u1) at (axis cs:-1.4,-1.3) {$\textcolor{mycolor4}{g}$}; \node (u1) at (axis cs:-0.2,0.2) {$\textcolor{mycolor2}{f}$}; \end{axis} \end{tikzpicture}
		\qquad \qquad
		\begin{tikzpicture}[scale=0.8] \begin{axis}[ width=2.5in, height=2.5in, scale only axis, xmin=-0, xmax=2.0, ymin=-0.0, ymax=2.0, ticks = none, ticks = none, axis background/.style={fill=white}, axis line style={draw=none} ] \addplot [color=mycolor1,solid,thick, fill = mycolor1!20!white,forget plot] table[row sep=crcr]{ 0 0\\
0 1\\
1 2\\	 2 1\\
1 0\\
0 0\\
}; \addplot [color=mycolor1,solid,thick,forget plot] table[row sep=crcr]{ 0.8015206866389119 0.7999200483580394 \\
0.8019146543158784 0.7998993627564008 \\
0.8024107604080449 0.799873330540932 \\
0.8030355268717795 0.7998405729516069 \\
0.8038223867518556 0.7997993576236457 \\
0.8048134988914829 0.7997475091210565 \\
0.806062046874736 0.7996822976350562 \\
0.8076351555634163 0.7996003010787001 \\
0.809617597368483 0.7994972351992968 \\
0.8121165112686588 0.7993677460358508 \\
0.8152674244436201 0.7992051595371545 \\
0.8192419540965267 0.7990011853508537 \\
0.8242576811941681 0.7987455775006224 \\
0.8305908335989786 0.7984257674033639 \\
0.8385925946880386 0.7980265112818004 \\
0.8487100538142075 0.7975296478143193 \\
0.8615129926588495 0.7969141687238057 \\
0.8777277332115833 0.7961570150409301 \\
0.8982788551331814 0.7952354207017245 \\
0.9243380273445426 0.79413241077506 \\
0.9573749889870049 0.7928485355196047 \\
0.9991958313742026 0.791425562418795 \\
1.0519327307569295 0.7899921541801014 \\
1.1179095655541025 0.7888471552589856 \\
1.199246488595043 0.7885986278023105 \\
1.2970128996494639 0.7903602166914636 \\
1.4098189519444153 0.7959336117202666 \\
1.5322160243082719 0.8077384073423377 \\
1.6542479179249738 0.8280790054346812 \\
1.7639030003179543 0.8575738626873127 \\
1.8520172740863572 0.8935755118811415 \\
1.9156724832613126 0.930264393312653 \\
1.9572410240415283 0.9611101601616978 \\
1.981413489825807 0.9821426125570012 \\
1.9934131091512022 0.9935031773445718 \\
1.9982079603758143 0.9982145168332485 \\
1.9996515898883274 0.9996518349261023 \\
1.9999556653526878 0.9999556692968918 \\
1.9999966919341592 0.9999966919560658 \\
1.9999998739471652 0.9999998739471969 \\
1.9999999979388465 0.9999999979388464 \\
1.9999999999883824 0.9999999999883824 \\
1.9999999999999827 0.9999999999999828 \\
2.0 1.0 \\
2.0 1.0 \\
2.0 1.0 \\
2.0 1.0 \\
2.0 1.0 \\
2.0 1.0 \\
2.0 1.0 \\
2.0 1.0 \\
}; \addplot [color=mycolor1,solid,thick,forget plot] table[row sep=crcr]{ 0.8009603115660523 0.7997600722384104 \\
0.8012090620159799 0.7996979724891757 \\
0.8015222794449439 0.7996198074387844 \\
0.8019166904670288 0.799521425575409 \\
0.8024133726752407 0.7993976053245092 \\
0.8030388928074117 0.7992417809014174 \\
0.8038267462586172 0.7990456989582359 \\
0.8048191797227421 0.7987989892163185 \\
0.8060695019863984 0.7984886286429569 \\
0.8076450182304781 0.7980982746482863 \\
0.8096307629457115 0.797607438485771 \\
0.812134258814783 0.7969904661162434 \\
0.8152915976103865 0.7962152914785425 \\
0.8192752291102848 0.7952419288964565 \\
0.8243039603687948 0.794020682140102 \\
0.8306558140849595 0.7924900769137877 \\
0.8386845686948105 0.790574589389379 \\
0.8488409836918831 0.7881823803452643 \\
0.861699836278371 0.7852035176557248 \\
0.8779937903744512 0.7815097006830861 \\
0.8986543866163524 0.7769575090261016 \\
0.9248581982497683 0.771399066176242 \\
0.9580706078063128 0.7647073298064309 \\
1.0000672400913821 0.7568287031020569 \\
1.0528882573492457 0.747883395972206 \\
1.1186380037594974 0.7383407586876767 \\
1.1989889281071353 0.7292892314973001 \\
1.2942409287399277 0.7227636510102498 \\
1.4019883995608584 0.7219341445947441 \\
1.5160665716150452 0.7307215635289576 \\
1.6272153724877003 0.7523979922233262 \\
1.7263607153126332 0.787496143908311 \\
1.808544221257428 0.8325352720130694 \\
1.8736543079532464 0.880861979738419 \\
1.9234884352890225 0.925047917224536 \\
1.9591713950469551 0.9593958134257301 \\
1.9816572428628116 0.9816766330717708 \\
1.993384736806289 0.9933856204494347 \\
1.9981844751463598 0.9981844931856546 \\
1.999645086672177 0.9996450868064222 \\
1.9999546021310168 0.9999546021312976 \\
1.9999965915657159 0.9999965915657161 \\
1.9999998691130751 0.9999998691130751 \\
1.9999999978388454 0.9999999978388453 \\
1.9999999999876683 0.9999999999876684 \\
1.9999999999999813 0.9999999999999815 \\
2.0 1.0 \\
2.0 1.0 \\
2.0 1.0 \\
2.0 1.0 \\
2.0 1.0 \\
}; \addplot [color=mycolor1,solid,thick,forget plot] table[row sep=crcr]{ 0.8002400717616105 0.7990403124332517 \\
0.8003022557354967 0.7987919269556245 \\
0.8003805542729939 0.7984792879677989 \\
0.8004791476992491 0.7980857937106874 \\
0.8006033032549009 0.7975905644708869 \\
0.8007596591006785 0.796967347123005 \\
0.8009565832932255 0.7961831433001192 \\
0.8012046279133658 0.7951964940095557 \\
0.8015171041335113 0.7939553387982544 \\
0.8019108112208898 0.7923943507377971 \\
0.8024069617237174 0.7904316300965586 \\
0.8030323568717745 0.7879646210443834 \\
0.8038208810067857 0.7848651002090208 \\
0.8048154018509851 0.7809730795231937 \\
0.8060701840337562 0.7760894801634779 \\
0.8076539439345078 0.7699674907444833 \\
0.8096536874351321 0.7623026594499677 \\
0.8121794617828485 0.752722053583742 \\
0.8153700818445211 0.7407733662743002 \\
0.8193996844726611 0.7259158489453277 \\
0.8244844766539992 0.7075166935999837 \\
0.8308880092699754 0.6848593790980904 \\
0.8389212990047358 0.6571749088111399 \\
0.8489305765831088 0.6237126655649656 \\
0.8612600448569387 0.5838727623581308 \\
0.8761710236942853 0.5374194419467476 \\
0.8936981113911688 0.48476975859949106 \\
0.9134422378097709 0.42728079046005574 \\
0.9343610007215083 0.3673383135475512 \\
0.9547171490089942 0.3079656991489156 \\
0.9723997291948999 0.2518473645295834 \\
0.9856657137333834 0.20024757246159033 \\
0.9939370498144786 0.15286115317698926 \\
0.9980277607784995 0.10917371938793886 \\
0.9995410165187973 0.07029142318795548 \\
0.9999300808509357 0.03909838520814076 \\
0.9999937308690112 0.01800027841890348 \\
0.9999997086807995 0.006571214124345083 \\
0.9999999940050677 0.0018122227501402631 \\
0.9999999999552781 0.00035478731966275126 \\
0.9999999999999064 4.539580792308283e-5 \\
1.0 3.408422666596551e-6 \\
1.0 1.3088690772857902e-7 \\
1.0 2.1611547448324278e-9 \\
1.0 1.2331497984168996e-11 \\
1.0 1.8467266624096254e-14 \\
1.0 5.133638251052602e-18 \\
1.0 1.712832928636948e-22 \\
1.0 3.961600680388522e-28 \\
1.0 3.182460953873484e-35 \\
1.0 3.720075976020882e-44 \\
}; \addplot [color=mycolor1,solid,thick, forget plot] table[row sep=crcr]{ 0.7994400720602481 0.7998399920229331 \\
0.7992951160019068 0.7997985593007249 \\
0.7991126409078549 0.7997463970854126 \\
0.7988829402192047 0.7996807263631942 \\
0.7985937988420752 0.7995980480578322 \\
0.7982298464142706 0.7994939562995687 \\
0.7977717447659594 0.7993629031826036 \\
0.7971951676254588 0.7991979023622873 \\
0.7964695203517314 0.7989901555233533 \\
0.7955563349101393 0.7987285815515929 \\
0.7944072600827226 0.7983992229310761 \\
0.7929615487102193 0.7979844971786677 \\
0.7911429224671133 0.7974622526664812 \\
0.7888556705406709 0.7968045775627081 \\
0.7859798127075072 0.7959762974063725 \\
0.7823651323498008 0.7949330806495657 \\
0.777823866578498 0.7936190522631028 \\
0.772121839970223 0.7919637938993928 \\
0.764967866594707 0.7898785876693101 \\
0.7560013610534039 0.7872517458771878 \\
0.7447783639162482 0.7839428760642235 \\
0.7307567257383842 0.7797759914750457 \\
0.7132822231090538 0.7745315568675228 \\
0.6915792562888867 0.7679379857233104 \\
0.6647530422013466 0.7596640127227743 \\
0.6318155626570646 0.7493151564167737 \\
0.5917554743065151 0.736440752548611 \\
0.5436819338757348 0.7205633735891905 \\
0.48707884370283294 0.7012495934131275 \\
0.42219496730792394 0.678246485872755 \\
0.35053864163768156 0.6517007306870544 \\
0.27531412839832325 0.6224350195447725 \\
0.20145766455889125 0.5921616636537566 \\
0.1348802752625155 0.5634038033567774 \\
0.08086899283838969 0.5389176147512719 \\
0.04229286892457165 0.5207170537773015 \\
0.018662433084882746 0.5092457257884179 \\
0.006658278197768967 0.5033181289976173 \\
0.0018188058932541262 0.5009085774329937 \\
0.0003550391793377504 0.5001774880876465 \\
4.539992971569891e-5 0.500022699449593 \\
3.4084459013856844e-6 0.5000017042200463 \\
1.3088694199134946e-7 0.5000000654434665 \\
2.1611547541736068e-9 0.5000000010805772 \\
1.2331497984473115e-11 0.5000000000061658 \\
1.8467266624096917e-14 0.5000000000000093 \\
5.1336382510525735e-18 0.49999999999999994 \\
1.7128329286369312e-22 0.49999999999999983 \\
3.961600680388468e-28 0.5000000000000006 \\
3.1824609538734295e-35 0.4999999999999994 \\
3.7200759760208177e-44 0.49999999999999933 \\
}; \addplot [color=mycolor1,solid,thick,forget plot] table[row sep=crcr]{ 0.799600080093308 0.8004000799066414 \\
0.7994965568128984 0.8005036967698851 \\
0.7993662440453368 0.8006341578561726 \\
0.7992022143007195 0.8007984226699503 \\
0.7989957516714943 0.8010052578582401 \\
0.79873589188485 0.8012657081100838 \\
0.7984088451148844 0.8015936907014966 \\
0.797997272308536 0.802006746677786 \\
0.7974793788832936 0.802526990751135 \\
0.7968277813956176 0.8031823137201903 \\
0.7960080930772744 0.8040079064160743 \\
0.7949771630778886 0.8050481939405629 \\
0.7936808922626655 0.806359294723632 \\
0.7920515365857849 0.8080121525323483 \\
0.7900043997384342 0.8100965332697815 \\
0.7874338146287322 0.8127261347199745 \\
0.7842083272325748 0.816045128470497 \\
0.7801650431821208 0.8202365392054034 \\
0.7751032079080973 0.8255329615213475 \\
0.7687773210787434 0.8322301974514897 \\
0.7608905351487413 0.8407044135161876 \\
0.7510899293119099 0.8514332338339343 \\
0.7389667719102595 0.8650205197658272 \\
0.724067524838948 0.8822228474094846 \\
0.7059256503006408 0.9039717502553297 \\
0.6841305613797686 0.9313777000525656 \\
0.658457135862355 0.9656868964975737 \\
0.6290811446737619 1.008138764027184 \\
0.5968858081490506 1.0596475780376997 \\
0.5637898858722487 1.120238235386797 \\
0.5328706469978659 1.1882796513814822 \\
0.5078871812904093 1.2598568359637663 \\
0.49193545613077727 1.3289731087806476 \\
0.48572960551830174 1.389115669672196 \\
0.4869175561789973 1.4356187864861467 \\
0.4914067282760475 1.467213970753633 \\
0.49572340207592774 1.48578994666022 \\
0.4983851761010123 1.4949787860670738 \\
0.49954901745917 1.4986338321805461 \\
0.49991138199501156 1.4997336418632505 \\
0.49998865233632805 1.4999659487645567 \\
0.499999147901594 1.4999974436583128 \\
0.49999996727828344 1.4999999018347823 \\
0.49999999945971124 1.499999998379134 \\
0.499999999996917 1.4999999999907514 \\
0.4999999999999954 1.4999999999999862 \\
0.5 1.5 \\
0.5 1.5 \\
0.4999999999999999 1.4999999999999998 \\
0.5 1.5 \\
0.5000000000000001 1.5000000000000002 \\
}; \addplot [color=mycolor1,solid,thick,forget plot] table[row sep=crcr]{ 0.7993603132834404 0.8013607102951685 \\
0.7991947847837753 0.8017132636014286 \\
0.7989864571741683 0.8021572364134206 \\
0.7987242844051656 0.8027163777144934 \\
0.7983943816047351 0.8034206308980771 \\
0.797979302862857 0.8043077635771535 \\
0.7974571398978111 0.8054254340365832 \\
0.7968004002115665 0.8068338154317387 \\
0.7959746159123604 0.8086089343738148 \\
0.7949366273203108 0.8108469271898042 \\
0.7936324805721677 0.8136694783456109 \\
0.7919948793086022 0.8172307853633345 \\
0.7899401441618903 0.8217264972017082 \\
0.7873646735089389 0.8274052008952613 \\
0.7841409899678204 0.834583180174753 \\
0.7801136463635393 0.8436633180339084 \\
0.7750956390829215 0.855159097985208 \\
0.7688666930705012 0.8697245134542939 \\
0.7611761169354042 0.8881899470751803 \\
0.7517553329332074 0.9116019228134227 \\
0.7403493186686129 0.9412594300403286 \\
0.7267826908801218 0.9787282813443154 \\
0.7110845754336228 1.025793204297248 \\
0.6937019122707333 1.0842715587895795 \\
0.6758166928597186 1.1555716082710483 \\
0.6597109057387432 1.2398841638504807 \\
0.6489494417062046 1.3350840812441684 \\
0.6479334385700718 1.435936774227057 \\
0.6604699080184676 1.5347937493915806 \\
0.6878641109758465 1.6244744851580613 \\
0.7280679877384533 1.7017323656754375 \\
0.7767578120615259 1.768061765741601 \\
0.8290440936160455 1.8268597038579593 \\
0.8800803853764755 1.8796865887667125 \\
0.9249373871487816 1.9248903222326643 \\
0.9593869602726229 1.9593836084563292 \\
0.9816762822209906 1.981676158527315 \\
0.9933856146306814 1.9933856126519487 \\
0.9981844931529437 1.9981844931419803 \\
0.9996450868063743 1.9996450868063584 \\
0.9999546021312975 1.9999546021312975 \\
0.9999965915657161 1.999996591565716 \\
0.9999998691130753 1.9999998691130754 \\
0.9999999978388452 1.9999999978388452 \\
0.9999999999876685 1.9999999999876685 \\
0.9999999999999815 1.9999999999999813 \\
1.0 2.0 \\
1.0 2.0 \\
1.0 2.0 \\
1.0 2.0 \\
1.0 2.0 \\
}; \end{axis} \end{tikzpicture}
		\caption{Newton polygons and sector decomposition from \Cref{ex:pentagon2}.}
		\label{fig:polygonsEx45}
	\end{figure}

	Let $\mathcal{F}$ be the normal fan of the hexagon $\mathcal{N}(f) + \mathcal{N}(g).$
	There are six cones $\sigma$ in
	$\mathcal{F}(2).$
	The tropical integral $\mathcal{I}^{\rm tr}$ is the sum of the
	six numbers $\mathcal{I}^{\rm tr}_\sigma$
	in  $\eqref{eq:detVU}.$
	 We~find
	$$ \mathcal{I}^{\rm tr} \,\,= \,\,
	1 + 2 + \frac{3}{2} + 1 + \frac{1}{4} + \frac{7}{2} \,\, = \,\,
	\frac{37}{4}.
	$$
	The surface $X_{> 0}$ is divided into six sectors ${\rm Exp}(\sigma).$ This 
	can be visualized via the moment map $X_{>0} \rightarrow P^\circ,$ as shown
	in Figure \ref{fig:polygonsEx45}. On each sector,
	we have a monomial map $q \mapsto x^{\sigma}(q)$ 
	with rational exponents, given  in Equation \eqref{eq:xsigmaq}.
	Using this map, we now apply Algorithm~\ref{alg:eins}.
	We draw $N=10000$ samples from the tropical density~$d^{\rm tr}_{f,g}.$
	The formula \eqref{eq:montecarlo} then
	gives the following approximate value for the classical~integral:
	$$ \mathcal{I} \,\, \approx \,\,  2.8677596477559826. $$

	We next apply  \Cref{prop:expdivi}. From
	\eqref{eq:M1M2} we get  $M_1 = 1/24$ and $M_2 = 10/7.$
	This implies that the standard deviation $\sqrt{\mathbb{E}[(\mathcal{I} - \mathcal{I}_N)^2]}$
	is at most  $0.132.$
	By comparing with a more accurate approximation, using numerical cubature for \eqref{eq:cubetrafo}, we find that the error is no larger than $0.005.$
	Repeating this experiment for a range of sample sizes~$N,$ we find
	that our approximation beats the generic bound \eqref{eq:generalbound} by  two orders of magnitude.
	This illustrates a phenomenon that is observed for many examples: the bounds~\eqref{eq:generalbound} are overly pessimistic.

	Finally, we use \Cref{alg:mufg} to sample from the 
	posterior distribution $d_{f,g}.$ From
	$N=100000$ candidate samples, $21808$ were accepted. The bound on the expected value of the acceptance rate in \Cref{prop:acceptancerate} is $M_1/M_2 \approx 0.03.$ Again, this is pessimistic. From the proof of \Cref{prop:acceptancerate} we see that the actual expected acceptance rate is 
	\[ \frac{1}{M_2} \cdot \frac{{\cal I}}{{\cal I}^{\, \rm tr}} \, \approx \,  0.22. \qed \] %
\end{example}

\section{Statistical Models} \label{sec5}
In this section, we
present several statistical models, some well-known and
others less so. They all have
a natural polyhedral structure which allows for  a parameterization $X_{>0} \rightarrow \Delta_m$
from a toric variety $X$.
This includes both toric models and linear models. We argue that this passage
to toric geometry
makes sense, also from an applied perspective, since
Bayesian integrals \eqref{eq:intonpol} can now be evaluated
using tropical sampling.
Such integrals depend on experimental data. We will study them in the
next section. 

The common parameter space for our models is the positive part of a projective toric variety $X = X_\Sigma$ of dimension $n$.
We assume that the fan $\Sigma$ is simplicial, so it is the
normal fan of a simple lattice polytope $P $ in $\RR^n.$ The polytope $P$
is not unique. There is one such polytope for each very ample divisor on $X.$
The vertex set $\mathcal{V}(P)$ is in  bijection with the maximal cones in the normal fan of $P.$
The prior distribution on $X_{>0}$ has a density that is 
a positive rational function. We obtain $1$ when integrating this density  against the~form
\begin{equation}
	\label{eq:omegax}
	\Omega_X \quad = \quad \sum_I {\rm det}(V_I) \bigwedge_{i \in I} \frac{{\rm d} x_i}{x_i}. 
\end{equation}

Fix the uniform distribution on the polytope $P.$ We  consider models of the form
\begin{equation}
	\label{eq:modelfrompolytope} P \, \rightarrow \,\Delta_m, \quad \,y \,\mapsto \,\bigl(\,p_0(y),\,p_1(y),\,\ldots,\,p_m(y)\, \bigr). 
\end{equation}
One task in \Cref{sec6} is to evaluate marginal likelihood integrals \begin{equation} \label{eq:intonpol}
	\int_P p_0^{u_0} p_1^{u_1} \cdots p_m^{u_m} \, \ {\rm d} y_1\wedge  \cdots \wedge {\rm d} y_n. 
\end{equation}
Here we use uniform priors on $P$. 
One still needs to divide by the volume of $P.$ We here ignore this factor for simplicity. 
While (\ref{eq:modelfrompolytope}) is fairly natural from a statistical perspective,
it seems that the construction in the next paragraph, namely lifting this to the
toric variety via the moment map, has not yet been considered in the statistics literature.

We lift \eqref{eq:modelfrompolytope} to the positive toric variety $X_{>0}$
by composing with the moment map $X_{>0} \rightarrow P^\circ,$ where $P^\circ$ is the interior of $P.$ We do this in two steps, by writing the moment map as $\varphi \circ \phi,$ where $\phi$
is the identification $ X_{>0} \simeq \RR^n_{>0}$  and $\varphi : \RR^n_{>0} \rightarrow P^\circ$ is the \emph{affine moment map}. The latter can be defined by a Laurent polynomial with positive coefficients $c_a \in \RR_{>0},$
\[ q \,\,\, = \sum_{a \in \mathcal{V}(P)} \!\! c_a \, t^a ~ \in \,\,\RR[t_1^{\pm 1}, \ldots, t_n^{\pm 1}]. \]
The map $\varphi$ sends $t \in \RR^n_{>0}$ 
to the following convex combination of $\mathcal{V}(P)$:
\begin{equation}
	\label{eq:affinemomentmap}
	\varphi(t)  \,\,=\,\, \frac{1}{q(t)} 
	\left (\,\theta_1(q(t)),\,
	\ldots\,,\, \theta_n(q(t))\, \right ) \,
	\,\,=\,\,\sum_{a \in \mathcal{V}(P)} \frac{c_a \, t^a}{q(t)} \cdot a .
\end{equation}
Here, $\theta_i$ denotes the $i$th Euler operator $t_i\partial_{t_i}.$
The {\em toric Jacobian} $(\theta_j(\varphi_i))_{i,j}$  of the map~$\varphi$ is the  {\em toric Hessian} $H$ of ${\rm log}(q(t)).$
This is the symmetric $n \times n$ matrix with entries
$ H_{i,j} =\theta_i\theta_j \left(\log q(t)\right).$
Since $\varphi$ is a diffeomorphism, the determinant of $H$ is nowhere zero on $\RR^n_{> 0}.$ Moreover, its denominator $q^n$ has positive coefficients, so there are no poles on~$\RR^n_{>0}$ either.
Recall that the columns of $V \in \ZZ^{n \times k}$ are the facet normals of 
the simple polytope~$P.$ This gives us a formula for
the density on $X_{>0}$ that represents
the uniform distribution~on~$P.$

\begin{proposition} \label{prop:liftform}
	The pullback of ${\rm d} y_1 \cdots {\rm d} y_n$ under the moment map  $X_{>0} \rightarrow P^\circ$
	is a positive rational function $r$ times the canonical form $ \Omega_X.$ We obtain $r(x)$ 
	from the toric Hessian $H(t)$ by replacing $t_1,\ldots,t_n$
	with the Laurent monomials in $x_1,\ldots,x_k$ given by the rows of~$V.$
\end{proposition}
\begin{example}
	For the coin model in the Introduction, with
	$n=3,$ $P = [0,1]^3,$ and $X = \PP^1\times \PP^1 \times \PP^1$, the desired  function $r$ 
	is the factor before $\Omega_X$ in Equation~\eqref{eq:unif}.
\end{example}

\Cref{prop:liftform}
means that  the integral \eqref{eq:intonpol} is the following integral over~$X_{>0}$:
\begin{equation}
	\label{eq:intonpol2}
	\int_{X_{>0}} p_0(x)^{u_0} p_1(x)^{u_1} \cdots p_m(x)^{u_n} \, r(x) \,\Omega_X ,
\end{equation}
where $p_i(x)$ arises from $p_i(y)$ by the above two-step substitution: we
first set $y = \varphi(t)$ and then we replace $t$ with the Laurent monomials in $x$ given by the rows of~$V.$

\begin{example} \label{ex:pentagonrevisit}
	The pentagon $P$ from \Cref{ex:pentagon1}  is the Newton polytope of
	$$ q(t_1,t_2) \,=\, t_2^{-1} + t_1^{-1}t_2^{-1} + t_1^{-1} + t_2 + t_1.$$
	We identify its interior $P^\circ$ with the positive quadrant $\RR^2_{>0}$ via the affine moment map
	$$ \varphi\colon \,\RR^2_{>0} \stackrel{\simeq}{\longrightarrow} P^\circ, \quad  (t_1, t_2) \, \mapsto\, 
	\frac{1}{q} \left (\theta_1(q), \theta_2(q) \right).$$
	Its toric Jacobian is the toric Hessian $H$ of ${\rm log}(q).$ 
	Its determinant is
	$$\det(H) \,\,=\,\,  \frac{4 t_1^3 t_2^2+4 t_1^2 t_2^3+9 t_1^2 t_2^2+4 t_1^2 t_2+4 t_1 t_2^2+t_1^2+8 t_1 t_2+t_2^2+1}{ (t_1 t_2)^{2} \cdot q(t_1,t_2)^{3}}. $$
	Turning rows of  \eqref{eq:Vmatrix} into monomials, we set
	$t_1 = x_1 x_2 x_3^{-1} x_4^{-1}$ and
	$t_2 = x_2^{-1} x_3^{-1} x_4 x_5.$
	This substitution turns $\det(H)$ into the rational function $r,$ as seen in
	\Cref{prop:liftform}. Writing $r = f/g$ as in
	\Cref{thm:convergent}, one sees that the  Newton polygons
	satisfy~the containment hypothesis.
	It is instructive to compute the area of the pentagon via
	$$   \int_{X_{>0}}\!\! r(x) \,\Omega_X \,\, = \,\,\int_P \! 1 \, {\rm d}y_1{\rm d}y_2 \,\,=\,\, \frac{5}{2}. $$
	The integrals are \eqref{eq:intonpol2} $=$
	\eqref{eq:intonpol} with $u_i=0.$
	The first is found numerically by~\Cref{sec4}.
\end{example}

We now turn to the statistical models associated to our polytope $P$.
We begin with the
{\em linear model} associated with our polytope $P.$
From now on we assume that $P$ contains the origin in its interior.
We thus have the inequality representation
$$ P \,\,\, = \,\,\, \bigl\{ \,y \in \RR^n \,\,|\,\, \langle v_i , y \rangle + \alpha_i \,\geq \, 0 \,\,
{\rm for} \,\, i = 1,2,\ldots,k \,\}, $$
where $\alpha_1,\ldots,\alpha_k $ are positive integers.
Vertices $q_I$ of $P$ are indexed by cones $I \in \Sigma(n).$
The vertex $q_I \in \ZZ^n$ is the unique solution to the $n$ linear equations
$\langle v_i , y \rangle \,=\, - \alpha_i $ for $i\in I.$ The following lemma helps to interpret the facet equations as probabilities.
\begin{lemma}
There exist $\gamma_i>0, i = 1, \ldots, k$ such that $\sum_{i=1}^k p_i(y) = 1$, with 
\begin{equation} \label{eq:pifacet}
	p_i(y) \,\, = \,\, \frac{1}{\gamma_i} \bigl( \alpha_i + \langle v_i , y \rangle \bigr).
\end{equation}
\end{lemma}
\begin{proof}
It suffices to find a positive vector $1/\gamma = (1/\gamma_1, \ldots, 1/\gamma_k)$ in the kernel of $V = [v_1~ \cdots ~v_k]$, scaled so that $(1/\gamma) \cdot \alpha = 1$.
Such a vector exists because the columns $v_i$ are the rays of a complete fan $\Sigma_P$. Indeed, $-v_i$ is a positive combination of the rays spanning the smallest cone of $\Sigma_P$ containing it. This gives a nonnegative vector $w_i \in \ker V$ with $i$-th entry $1$ for each $i$. Pick an interior point $y \in P^\circ$. Since $\alpha_j > - \langle v_j, y \rangle$ for all $j$, we have $w_i \cdot \alpha > 0$. 
We conclude that the vector
$1/\gamma = \frac{1}{\sum_{i = 1}^k w_i \cdot \alpha } \sum_{i=1}^k w_i$ is positive.
\end{proof}

The states of the linear model are the $k$ facets of $P.$ The probability of the $i$-th facet is given by \eqref{eq:pifacet}.
The probabilities are nonnegative
precisely on the polytope $P.$ The linear model is the image of the
resulting map~$P \rightarrow \Delta_{k-1}.$ 
While the states in the linear model are the facets of our simple polytope $P,$ we now
introduce a variant where the vertices $q_I$ serve as the states.
Their number is $m+1 = |\Sigma(n)|.$

The following polynomial in $n$ variables is known as the {\em adjoint} of
the polytope~$P$:
$$ A(y) \quad = \quad  \sum_{I \in \Sigma(n)}  \! |{\rm det}(\widetilde{V}_I)| \cdot 
\prod_{i \not\in I} \biggl( 1+  \frac{1}{\alpha_i} \langle v_i,y \rangle \biggr) . $$
Here the matrix $\widetilde{V}$ is obtained from $V$ by scaling the $i$th column with $\alpha_i^{-1}.$
This formula looks like the canonical form $\Omega_X$ on the
toric variety $X.$ Namely, we replace
$x_i$ in \eqref{eq:omegax}
by the $i$-th facet equation and we clear denominators.
The following differential form on $P$ is the pushforward of $\Omega_X$ under
the moment map $X_{>0} \rightarrow P$:
$$ \Omega_P \,\, = \,\, \frac{A}{\prod_{i=1}^k \left ( 1 + \frac{1}{\alpha_i} \langle v_i, y \rangle \right) } \,{\rm d} y_1 \cdots {\rm d} y_n .$$
Arkani-Hamed, Bai, and Lam \cite[Theorem 7.2]{arkani2017positive}
proved that $\Omega_P$ is the canonical form of the pair $(\PP^n,P).$
The adjoint $A$ endows $P$ with the structure of a positive geometry.

Each summand of $A$ has degree $k-n$. The adjoint has
degree $k-n-1,$ since the highest degree terms cancel.
Consider the summand indexed by the cone $I \in \Sigma(n)$:
\begin{equation}
	\label{eq:wachspressmap}
	p_I(y) \quad = \quad \frac{|{\rm det}(\widetilde{V}_I)|}{A(y)} \prod_{i \not\in I} \biggl( 1+ \frac{1}{\alpha_i} \langle v_i,y \rangle \biggr) . 
\end{equation}
These products of $k-n$ affine-linear forms satisfy the following remarkable identities:
$$ 
\sum_{I \in \Sigma(n)} \! p_I(y) \,=\, 1 \qquad {\rm and} \qquad
\sum_{I \in \Sigma(n)} \! p_I(y) \,q_I \,=\, y.
$$
These identities tell us that the $p_I(y)$ serve as
barycentric coordinates on $P.$ They express 
each point $y $ in the polytope $ P$ canonically as a convex
combination of the $m+1$ vertices $q_I.$
The resulting statistical model with state space $\Sigma(n)$ is the map
$$\,P \longrightarrow \Delta_m, \quad y \mapsto \bigl(\, p_I(y) \,\bigr)_{I \in \Sigma(n)}. $$
We call this the {\em Wachspress model} on the polytope $P$. We believe that this model, unlike the
linear model on $P$, 
has not yet been considered in the statistics literature.

\begin{example}[Pentagon] \label{ex:pentagonwachspress}
	The pentagon in \Cref{ex:pentagon1} matches \cite[Figure~8]{arkani2017positive}.
	Here, $n=2,$ $ k=m+1=5,$ and $P$ is the set where
		the following are nonnegative:
	\begin{equation}
		\label{eq:pentagonlinear}  \ell_1 = 1+y_1 , \,\,\,
		\ell_2 = 1+y_1-y_2 , \,\,\,
		\ell_3 = 1 - y_1 - y_2 ,\,\,\,
		\ell_4 = 1 - y_1 + y_2 ,\,\, \,
		\ell_5 = 1+y_2.
	\end{equation}     
	Here, $P$ was shifted so that the interior point in Figure \ref{fig:pentagon} is the origin.
	The vertices~are
	$$                      y_{12} \,=\, (-1, 0) \, , \quad
	y_{23} \,=\, (0, 1) \, , \quad
	y_{34} \,=\, (1, 0) \,, \quad
	y_{45} \,=\, (0, -1) \, , \quad
	y_{51} \,= \,(-1, -1). $$
	The linear model is the map $P \rightarrow \Delta_4,\, y \mapsto \frac{1}{5} (\ell_1(y),\ldots,\ell_5(y)).$ 
	Its states are the  edges of the pentagon $P.$
	The distributions in this model are the points $p \in \Delta_4$ that satisfy
	$$ 
	2p_1 - 2p_2 +p_3 - p_4  \,\,=\,\,
	- p_2 - 2p_4 + p_3 + 2p_5  \,\,=\,\, 0.
	$$
	We next describe the Wachspress model.
	The adjoint  of $P$ is the quadratic polynomial
	$$ A \,= \, 7 + 2 (y_1+y_2) - (y_1-y_2)^2 \,\,= \,\,
	\ell_1 \ell_2 \ell_3 + \ell_2 \ell_3 \ell_4 + 
	\ell_3 \ell_4 \ell_5 + 2 \ell_4 \ell_5 \ell_1 + 2 \ell_5 \ell_1 \ell_2 . $$
	The states of the model are the five vertices of the pentagon $P.$
	Their probabilities are
	\begin{equation}
		\label{eq:pentagonwachspress} (p_{45}, p_{51}, p_{12}, p_{23},p_{34}) \, = \,
		\frac{1}{A} \bigl(\,  \ell_1 \ell_2 \ell_3 ,\, \ell_2 \ell_3 \ell_4,\,
		\ell_3 \ell_4 \ell_5 , \, 2 \ell_4 \ell_5 \ell_1 , \,2 \ell_5 \ell_1 \ell_2 \,\bigr). 
	\end{equation}
	Each $p_{ij}$ is a rational function with cubic numerator and quadratic denominator.
	This defines the Wachspress model $P \rightarrow \Delta_4.$
	Its distributions are points $p \in \Delta_4$ that satisfy
	$$
	2 p_{12} p_{45}+2 p_{23} p_{45}-p_{23} p_{51}-2 p_{34} p_{51}
	\,=\, 2 p_{12} p_{34}-2 p_{23} p_{45}-p_{23} p_{51}+p_{34} p_{51} 
	\, =\, 0.
	$$          
	Geometrically, this is a del Pezzo surface of degree four
	in $\PP^4,$ obtained by blowing up $\PP^2$ at five points.
	These points are the intersections of edge lines outside $P$.
\end{example}

We now turn to {\em toric models}. In algebraic statistics \cite{Sullivant}, these are
models parameterized by monomials. We recast them in the setting of \Cref{sec2}.
Fix a degree $\gamma \in {\rm Cl}(X)$. Let $Z$ be a homogeneous polynomial
of degree $\gamma$ with positive~coefficients, 
$$ Z \,\, = \,\, c_0 x^{a_0} \,+ \, c_1 x^{a_1} \,+ \,\cdots \,+\, c_m x^{a_m} \,\, \, \in \, S. $$
We divide each of the summands by $Z$ to get rational functions $p_i$ of degree zero on~$X$:
$$ p_i \,\, = \,\, \frac{c_i x^{a_i}}{Z}\, , \qquad \hbox{for} \,\,\, i = 0,1,\ldots,m . $$
These functions are positive on $X_{>0}$ and their sum is equal to $1.$
The toric model of $Z$ is the resulting map  $X_{>0} \rightarrow \Delta_m $
into the probability simplex. In this manner, we identify toric models on $X$
with homogeneous positive polynomials $Z$ in the Cox ring.

The model is especially nice when  the degree $\gamma$ is ample 
and $Z$ uses all monomials of degree $\gamma.$ In that case, 
the Newton polytope $P = \mathcal{N}(Z)$ is simple and we have $\mathcal{F} = \Sigma.$
This simplifies the combinatorics and hence is a favorable situation for tropical sampling.

\begin{example} \label{ex:oneforeach}
	Let $\gamma$ be the very ample degree for the
	pentagon in \Cref{ex:pentagon1}.
	A general polynomial of degree $\gamma$ has six terms, one for each lattice point in Figure~\ref{fig:pentagon}:
		$$
		Z\,\, = \,\,
		c_0  x_2 x_3^3 x_4 \,+\,
		c_1  x_1 x_2^2 x_3^2 \,+ \,
		c_2 x_3^2 x_4^2 x_5 \,+ \,
		c_3 x_1 x_2 x_3 x_4 x_5 \,+ \,
		c_4 x_1^2 x_2^2 x_5 \,+ \,
		c_5 x_1 x_4^2 x_5^2.
		$$
	The toric model is  the map $X_{>0} \rightarrow \Delta_5 $ given by
	the six terms. Geometrically, up to scaling the coordinates by the $c_i > 0,$
	this is the  embedding of $X$ into $\PP^5$ 
	given by $\gamma.$
\end{example}

\begin{remark} \label{rmk:tensors}
	Let $P$ be a product of standard simplices, so the toric variety $X$ is a
	product of projective spaces. For $\gamma =(1,1,\ldots,1),$ 
	the line bundle $\mathcal{O}(\gamma)$
	is very ample. This
	 line bundle defines the Segre embedding of $X$. 
	Here, the toric model coincides with the Wachspress model.
	Each distribution in this model is a tensor of rank one
	\cite[Section 16.3]{Sullivant}. 
	Its mixture models encode tensors of higher~rank.
\end{remark}

The setting of \Cref{sec2} is convenient for working with {\em mixture models} \cite[Section 14.1]{Sullivant}.
Given any model $p:X_{>0} \rightarrow \Delta_m,$ its $r$-th mixture model
lives on the toric variety $X^r \times \PP^{r-1}.$ The parameter space
$(X^r \times \PP^{r-1})_{> 0} = X_{>0}^r \times \PP^{r-1}_{>0}$
is mapped into the probability simplex $\Delta_m$ by the secant map.
Geometrically, the mixture model is the $r$th secant variety  of ${\rm im}(p)$.
For more information see
\cite[Definition 14.1.5]{Sullivant} 

Mixture models of toric models  play an important role in applications.
Going beyond Remark \ref{rmk:tensors}, consider the model of
symmetric tensors of nonnegative rank $\leq r.$ In statistics, this is known as the
model of conditional independence for identically distributed random variables.
We refer to \cite{LSZ} for Bayesian integrals 
and to \cite[Section 5]{sturmfels2020likelihood} for likelihood inference.

The model in the Introduction is
the $r=2$ mixture of a toric model on $X = \PP^1.$
We conclude this section with a case study
of this coin model from the perspective of \Cref{sec3}.
The rational functions in \eqref{eq:runningex1} have distinct numerators $P_0,\ldots,P_m$
but the same denominator  $Q = (x_0+x_1) (s_0+s_1)^m (t_0+t_1)^m.$ 
The Minkowski sum of their Newton polytopes is a \mbox{$3$-dimensional} polytope in $\RR^6$.
In symbols, this is
\begin{equation}
	\label{eq:NewtonPolytope}
	\mathcal{N} (Q) \,+\,
	\mathcal{N}(P_0) \, + \,
	\mathcal{N}(P_1) \, + \, \cdots \,+\,
	\mathcal{N}(P_m) ,
\end{equation}
The normal fan $\mathcal{F},$ of this polytope, which
lives in a quotient space $\RR^6/\RR^3,$ is an essential ingredient for the algorithms
in Sections \ref{sec3} and \ref{sec4}. We now compute this.

\begin{theorem} \label{thm:nicepolyope}
	The Newton polytope \eqref{eq:NewtonPolytope}
	has $8 (m+1)$ vertices, $ 14 m + 12$ edges and $6(m+1)$ facets.
	Each of the eight vertices of the cube $\mathcal{N}(Q)$ is
	a summand of $m+1$ vertices.
	Among the $6(m+1)$ facets, four are pentagons, two are $2(m+1)$-gons and the remaining ones are quadrilaterals.
\end{theorem}

\begin{figure}[ht] 
	\centering
	\includegraphics[scale=0.51]{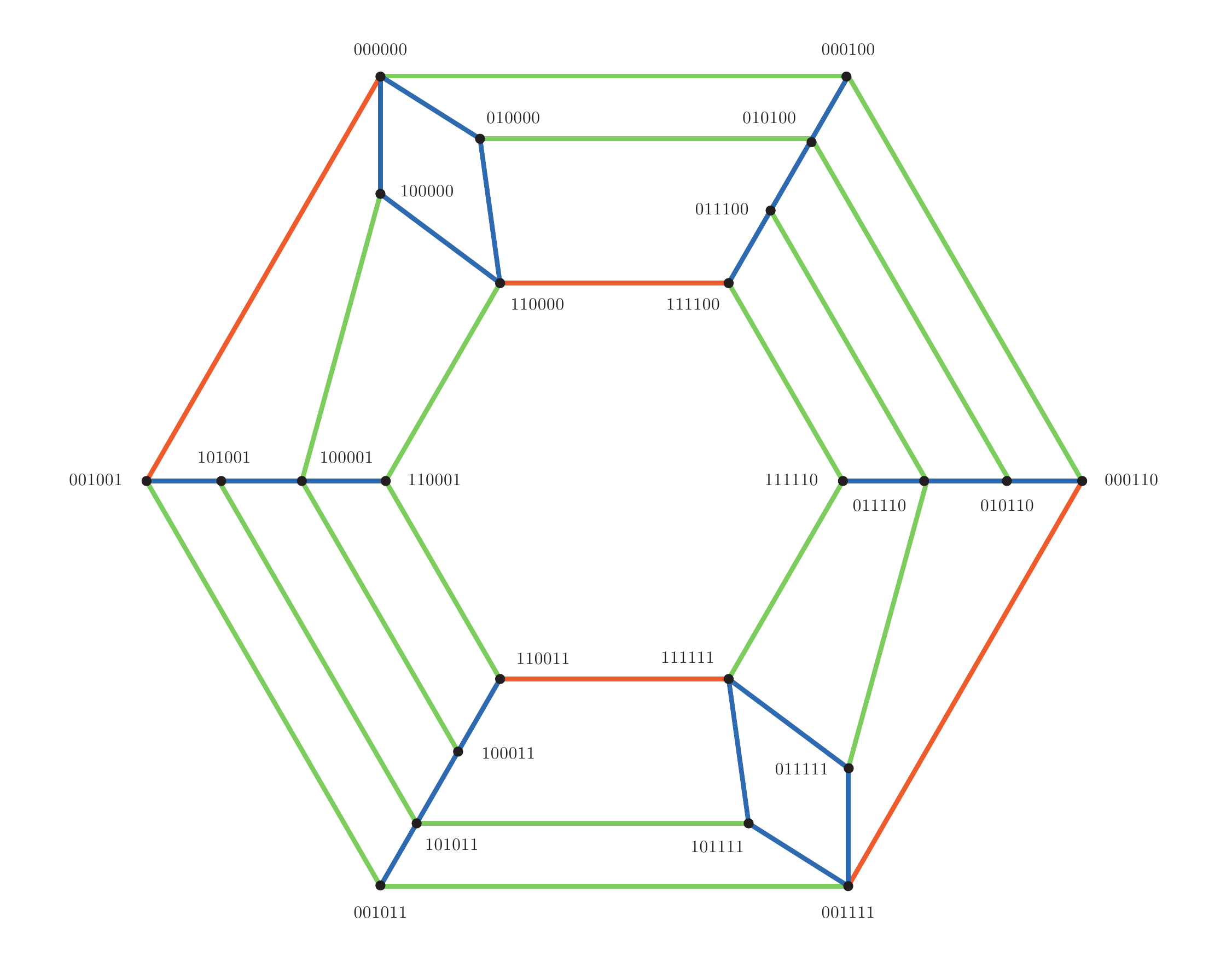} 
	\vspace{-0.2cm}	
	\caption{Schlegel diagram of the polytope \eqref{eq:NewtonPolytope} for $m=2.$
		The $24$ vertices are labeled by binary strings. There are four special edges (orange) and $36$ regular ones: $16$ of class $1$ (green) and $20$ of class $2$ (blue). Among the $18$ facets,
		we see two hexagons and~four~pentagons.
		\label{fig:eins}}
	\vspace{-0.2cm}		
\end{figure}

\begin{proof}[Sketch of Proof]
	Consider a generic vector $w$ in $\RR^6$ that assigns weights to the six Cox coordinates.
	The leading monomial $x_i s_j^m t_k^m$ of~$Q$ is determined
	by the signs of the quantities
	\begin{equation}
		\label{eq:threeineq}
		w(x_0) - w(x_1), \,\,w(s_0) - w(s_1), \,\,w(t_0) - w(t_1). 
	\end{equation}
	We record this leading monomial in the binary string $ijk.$
	Each of these eight choices
	allows for $m+1$ consistent choices of leading monomials
	from the tuple $(P_0,\ldots,P_m).$ Indeed, the leading monomial of $P_{\ell}$ coincides with that 
	of $\, \tilde P_\ell \,\, = \,\, x_0 s_0^\ell s_1^{m-\ell} t_0^m \, + \, x_1 t_0^\ell t_1^{m-\ell} s_0^m.$
	The line segments $\mathcal{N}(\tilde P_{\ell})$ lie in translates of a common $2$-dimensional
	subspace in $\RR^6,$ and their Minkowski sum is a $(2m+2)$-gon. Precisely half of its
	vertices are compatible with the inequalities \eqref{eq:threeineq}. These
	$m+1$ vertices become vertices of \eqref{eq:NewtonPolytope}, and they are
	all the vertices.

	We encode each vertex of~\eqref{eq:NewtonPolytope} by a binary string of length $m+4,$ starting with $ijk.$
	The other $m+1$ entries indicate the leading terms of the $P_{\ell}.$ Namely, we write $0$ if the 
	the following expression is positive, and we write $1$ if it is negative:
	$$ w(x_0) - w(x_1)\, +\, (m-\ell)(w(s_0) - w(s_1))\,+\, (m-\ell)(w(t_0) - w(t_1)) $$
	 With this notation, here is the list of all $8(m+1)$ vertices of our polytope:
	$$ \begin{matrix}
		000 \,1^{{\ell}} 0^{m-{\ell}+1}\, , & 
		010 \,1^{{\ell}} 0^{m-{\ell}+1}\,, & 
		100  \,0^{{\ell}+1} 1^{m-{\ell}}\,, & 
		110 \,0^{{\ell}+1} 1^{m-{\ell}}\,,  \\
		001\, 0^{{\ell}} 1^{m-{\ell}+1}\,, & 
		101 \,0^{{\ell}} 1^{m-{\ell}+1}\,, &
		011 \,1^{{\ell}+1} 0^{m-{\ell}}\,, & 
		111 \,1^{{\ell}+1} 0^{m-{\ell}}\,, \end{matrix} \quad \hbox{for} \,\,{\ell}=0,1,\ldots,m.
	$$
	Any pair of such strings that differs in precisely one entry is an edge of \eqref{eq:NewtonPolytope}.
	This accounts for all but four of the edges. These special edges are pairs of strings that differ in two~positions:
	$$ \begin{matrix}
		[000 \,0^{m+1} \,,\, 001 \,0^m 1], &
		[000 \,1^m 0 \,,\, 001 \,1^{m+1}],  \\
		[110 \,0^{m+1}\,,\, 111\, 1 0^m ], &
		[110 \, 0 1^m \,,\, 111 \, 1^{m+1}]. \end{matrix}
	$$
	The other $14m+8$ {\em regular edges} come in two classes.
	In the first class, the initial triple $ijk$ in the binary string is fixed. 
	For each initial triple $ijk$ there are $m$ such edges, for a total of $8m$ edges.
	The remaining $6m+8$ edges correspond to a sign change in 
	$w(x_0) - w(x_1),$ $w(s_0) - w(s_1)$ or $w(t_0) - w(t_1).$
	Here, the terminal $m+1$ letters in the binary string is fixed.
	If that string is $0^{m+1}$ or $1^{m+1}$ then there are four edges
	which form a square facet of our polytope, namely the square $\,00 0^{m+2}, \,01 0^{m+2}, \,11 0^{m+2} ,\, 10 0^{m+2}, \,00 0^{m+2}$
	and the square $\,00 1^{m+2}, \,01 1^{m+2}, \,11 1^{m+2} ,$ $10 1^{m+2},\,00 1^{m+2}.$
	For each of the other $2m$ terminal strings, there are only three edges which form a $3$-chain, for instance
	$ 001 0^m 1,$
	$101 0^m 1,$ 
	$100 0^m 1,$
	$110 0^m 1.$
	This accounts for all $\,4 + 8\cdot m + (4+4) + (2m) \cdot 3 \,= \, 14m+12\,$ edges of our polytope \eqref{eq:NewtonPolytope}.

	We now discuss the $6m+6$ facets. First, there are two centrally symmetric
	$2(m+1)$-gons. They are formed by all strings that start with $00$ or $11$ respectively.
	There are precisely two other facets adjacent to both $2(m+1)$-gons, namely the two squares facets
	mentioned above. Adjacent to these two squares and to the two big facets are the four pentagons, which are	
	\begin{footnotesize}
		$$ \hspace*{-3mm}\begin{matrix} & 000 0^{m+1}, 1 00 0^{m+1},  100 0^m 1, 101 0^m 1, 001 0^m 1  & &
			001 1^{m+1}, 011 1^{m+1} , 011 1^m  0, 010 1^m 0, 000 1^m 0,   \\ &
			111 1^{m+1} , 101 1^{m+1}, 101 0 1^m ,100 0 1^m, 110 0 1^m & &
			11 0 0 ^{m+1} , 010 0^{m+1} , 010 1 0^m, 011 1 0^m , 111 1 0^m.
		\end{matrix}$$ 
	\end{footnotesize} 
Each pentagon contains one of the four special edges. The remaining  facets are 
	quadrilaterals. They come in six strips of $m$ facets. 
	See Figure~\ref{fig:eins} for the case $m=2$.
\end{proof}

\section{Marginal Likelihood Integrals}\label{sec6}
We now come to applications of our results to Bayesian statistics.
We fix a statistical model
$X_{> 0} \rightarrow \Delta_m$ that is specified by $m+1$
rational functions $p_i$ on the toric variety $X.$
These rational functions  sum to $1.$
We write $p_i(x) = q_i(x)/r_i(x),$ where the numerator
$q_i$ and the denominator~$r_i$ are homogeneous polynomials
with positive coefficients that have the same degree in ${\rm Cl}(X).$
We further assume that we are given a rational function $f/g$ that defines
a probability distribution
$\mu_{f,g}$ on $X_{>0}$ as in \eqref{eq:classicaldensity}.
This serves as the prior distribution for Bayesian inference.

\begin{remark}
For the prior distribution on $X_{>0}$ we can choose any 
two positive polynomials $f$ and $g$ of the same
degree in ${\rm Cl}(X)$ such that the relevant
integrals converge. This is ensured by the 
hypothesis on Newton polytopes in Theorem  \ref{thm:convergent}.
\end{remark}

The data comes in the form of $U$ samples from the state space $\{0,1,\ldots,m\}.$
We write $u_i$ for the number of samples that are in state $i.$
We assume that the integers $u_i$ are positive. 
They satisfy $u_0+u_1 + \cdots + u_m = U.$ 
Given $u = (u_0,u_1,\ldots,u_m),$ the {\em likelihood function} $L_u$~is
\begin{equation}\label{eq:LF}
	L_u(x) \,\, \coloneqq \,\, \frac{\prod_{i=0}^m q_i(x)^{u_i} } {\prod_{i=0}^m r_i(x)^{u_i} } .
\end{equation} 
This is the probability of observing the data vector $u,$
assuming that the model, the parameters~$x,$ and the sample size $U$ are fixed.
The Multinomial Theorem implies 
$$ \qquad \sum_{|u| = U}  \frac{U!}{u_0! u_1 ! \cdots u_m !} \cdot L_u(x)  \,\,\,=\,\, \,1
\qquad \hbox{for all $x \in X_{>0}.$} $$

We are interested in the {\em posterior density} on $X_{>0}.$ Up to a constant factor,
this~is
\begin{equation}\label{eq:postprior}
	d_u \,\, := \,\, L_u \cdot d_{f,g}.
\end{equation}
The {\em marginal likelihood integral} $\,\mathcal{I}_u\,$ is the  integral  of \eqref{eq:LF}
against $\mu_{f,g},$ i.e.,
\begin{equation}
	\label{eq:MLI}
	\mathcal{I}_u \,\,\, := \,\,\,
	\int_{X_{>0}} \!\! L_u(x) \, \mu_{f,g} \, \, = \, \,  \frac{1}{\mathcal{I}} \int_{X_>0} \!\! L_u(x) \,  \frac{f(x)}{g(x)}\, \Omega_X  .
\end{equation}
We shall evaluate $\mathcal{I}_u$ using the methods in
\Cref{sec4}, but with $f/g$ replaced by
$L_u\cdot f/g.$ 
Also of interest is sampling  from the posterior density $d_u,$
by way of \Cref{alg:mufg}.

The Newton polytope of the integrand
$\,L_u\cdot f/g\,$ admits the decomposition
\begin{equation}
	\label{eq:NNNN} \mathcal{N}(f) \,+\, \mathcal{N}(g) \, +\, \sum_{i=0}^m u_i \,\mathcal{N}(q_i) \,+\,
	\sum_{i=0}^m u_i \,\mathcal{N}(r_i). 
\end{equation}
The normal fan of (\ref{eq:NNNN}) is independent of the data $u$ since 
$u_0,\ldots,u_m$ are positive.
As before, we let $\mathcal{F}$ be a simplicial refinement of the normal fan
of the polytope in \eqref{eq:NNNN}.
We  note the following fact, which is important for the applicability of our method.

\begin{observation}
	\label{obs:goodnews}
	The simplicial fan $\mathcal{F}$  is independent of the data $u.$
	It is computed from the statistical model.
	This is done in an offline step that is carried out only once per model.
\end{observation}

We point out that we may allow some of the $u_i$ to be zero. In this case, the true fan is a coarsening of ${\cal F}$. One may still use ${\cal F}$ in our method, at the cost of having more sectors. 

The computation of $\mathcal{F}$ is expensive when the dimension 
$n$ gets larger. Observation~\ref{obs:goodnews} means that the running time
of the algorithms in Sections \ref{sec3} and \ref{sec4} is fairly independent of $u.$
For instance, consider the computation of the sector integrals $\mathcal{I}^{\rm tr}_\sigma$
using the formula in \eqref{eq:detVU}. The data  $u$ do not appear in the numerator, but
they do enter in the denominator. Namely, the monomial $x^{-\delta_\sigma}$ that 
represents the tropicalized  integrand $\,L_u^{\rm tr} f^{\rm tr} / g^{\rm tr}\,$ on 
${\rm Exp}(\sigma)$ satisfies
$$ \delta_\sigma \,\, = \,\,\nu_g - \nu_f \,+\, \sum_{i=0}^m u_i (\nu_{r_i} - \nu_{q_i}). $$
In an offline step, done once per model,
we precompute the $ k \times (m+1)$ matrix of 
inner products $w_\ell \cdot (\nu_{r_i} - \nu_{q_i}).$
In the online step, with data $u,$ we   evaluate~\eqref{eq:detVU} rapidly.

One point that does depend on the data $u$ is the accuracy of the
approximation in~\eqref{eq:montecarlo}. The bounds $M_1$ and $M_2$ for the
function $h $ scale exponentially in $U,$ and hence so does the right hand side of
\eqref{eq:generalbound}.  The quality of the estimate $\mathcal{I}_N$  can decrease a lot for larger~$U$, so that many more samples are needed to obtain an accurate approximation.
We observed this phenomenon in our computations.
This  issue requires further study.

We now present computational experiments with the models
we saw in Sections~\ref{sec1} and~\ref{sec5}. This material
is made available at {\tt MathRepo} \cite{mathrepo}.
Our readers can try it out.
Our implementation is in {\tt Julia}. It uses {\tt Polymake}~\cite{gawrilow2000polymake}
for polyhedral~computations.

\begin{example}[Coin model]
	Consider the coin model from the Introduction.
	We begin with $m=2.$ Fix $U=5$ and  $u = (u_0,u_1,u_2) = (2,1,2).$
	The marginal likelihood integral \eqref{eq:MLI} is a rational number.
	Using symbolic computation as in  \cite{LSZ},  we find 
	$$ \mathcal{I}_u \,\,=\,\, \frac{2267}{1559250}\,\, \approx \,\, 0.001454. $$
	We reproduce this number using tropical sampling.
	The Newton polytope \eqref{eq:NewtonPolytope} of the integrand is shown in Figure~\ref{fig:eins}.
	Its normal fan has $24$ maximal cones, one for each vertex. But,
	this fan is not simplicial since eight of the vertices  are $4$-valent.
	We turn \eqref{eq:NewtonPolytope}
	into a simple polytope by a small displacement of the facets.
	The resulting normal fan is simplicial, and it has $32 = 24+8$ maximal cones.
	That simplicial fan $\mathcal{F}$ is used for the sector decomposition.
	The right hand side in \eqref{eq:secdecomp} has $32$ summands, one
	for each cone  $ \sigma \in \mathcal{F}(3).$
	The values of the $32$ tropical sector integrals $\mathcal{I}^{\rm tr}_\sigma$ are the
	rational numbers $\,\frac{1}{280},\frac{1}{280},\frac{1}{280},\frac{1}{280}, \frac{1}{120},\ldots, \frac{1}{8},\frac{1}{7},\frac{1}{7}.$ Their sum equals
	$\mathcal{I}^{\rm tr} = \frac{40}{21} = 1.9047.$
	This gives the discrete probability distribution used in
	Step 1 of \Cref{alg:eins}.
	Numerical evaluation of \eqref{eq:montecarlo} with sample size $N= 50000$ yields
	$$ \mathcal{I}_N \,\, = \,\, 0.001486. $$
	We validated our method with a range of experiments
	for larger values of $U,m,N.$
\end{example}

\begin{example}[Pentagon models] 
	\label{ex:pentagonmodel} We  revisit the linear model and the Wachspress model from 
	\Cref{ex:pentagonwachspress}. Their common parameter space is the pentagon $P$
	with uniform prior.
	This is lifted to the toric surface $X_{>0}$ with density $d_{f,g}$ given by
	the homogeneous polynomials $f$ and $g$   in Example  \ref{ex:pentagonrevisit}.
	The coordinates of the two models are obtained from the polynomials in $y$ seen in
	\Cref{ex:pentagonwachspress}. We first substitute
	$y_1 = \frac{t_1}{q} \frac{\partial q}{\partial t_1}$
	and $y_2 = \frac{t_2}{q} \frac{\partial q}{\partial t_2},$ and then we set
	$t_1 = x_1 x_2 x_3^{-1} x_4^{-1}$ and
	$t_2 = x_2^{-1} x_3^{-1} x_4 x_5.$
	In each case, this yields a   rational function $L_u f/g$ that is homogeneous of degree zero in $x,$
	and satisfies the hypothesis in \Cref{thm:convergent}.

	The likelihood functions for the linear model and the Wachspress model look similar,
	but there is a crucial distinction. The latter also involves the adjoint $A.$
	This means that statistical inference is different for the two models. For instance,
	the maximum likelihood (ML) degree of the linear model on $P$ equals five, while the
	ML degree of the Wachspress model on $P$ equals eight.
	Recall, e.g.~from  \cite{HKS, sturmfels2020likelihood, Sullivant},
	that the {\em ML degree} of an algebraic statistical model is the
	number of complex critical points of the likelihood function of that model for general data.

	Consider the linear model with 
	$(u_1,u_2,u_3,u_4,u_5) = (20,16,10, 15, 23) .$ With \eqref{eq:pentagonlinear},
	\begin{equation}
		\label{eq:Iu1}
		\mathcal{I}_u \,\,\, = \,\,
		\frac{1}{5^{84}} \cdot \frac{2}{5} \cdot  \int_P \ell_1^{20} \ell_2^{16} \ell_3^{10}  \ell_4^{15} \ell_5^{23} \,
		{\rm d}y_1 {\rm d}y_2 .
	\end{equation}
	For the approximation by tropical sampling, we note that
	the Newton polygon of $L_u f/g$ has seven vertices, so its
	normal fan $\mathcal{F}$ has  $|\mathcal{F}(2)| = 7.$ 
	We find that ${\cal I}_u \approx 9.652 \cdot 10^{-60}.$

	We now compare this to the Wachspress model,
	where the probabilities are products of the linear forms, as shown in
	\eqref{eq:pentagonwachspress}.
	We pick the
	data $(u_{123},u_{234},u_{345},u_{451},u_{512}) = (2,3,5,7,11)$ 
	in order to match the exponents of the linear factors in the respective likelihood functions.	
	The marginal likelihood integral for the Wachspress model equals
	\begin{equation}
		\label{eq:Iu2}
		\mathcal{I}_u \,\,\, = \,\,  2^{13} \cdot \frac{2}{5} \cdot
		\int_P \ell_1^{20} \ell_2^{16} \ell_3^{10}  \ell_4^{15} \ell_5^{23} A^{-84} \,
		{\rm d}y_1 {\rm d}y_2 .
	\end{equation}
	where $A = 7 + 2 (y_1+y_2) - (y_1-y_2)^2$ is the adjoint.
	Again, the Newton polygon of $L_u f/g$ has seven vertices, so its
	normal fan $\mathcal{F}$ has  $|\mathcal{F}(2)| = 7.$  We find that ${\cal I}_u \approx 1.218 \cdot 10^{-66}.$
\end{example}

We next illustrate how our techniques can be applied to 
Bayesian model selection.

\begin{example}[Bayes factors]	As before, let $\mu_{f,g}$ denote the prior arising from the toric~Hessian. 
	We consider the data  $u = (u_0,u_1,\ldots,u_5) = (1,2,4,8,16,32).$
	We wish to decide between two models with $n=2$ and $m=5.$
	The two competitors are toric models with $p_i$ and $Z$ as in
	\Cref{ex:oneforeach}, with different coefficient vectors
	\mbox{$c = (c_0,c_1,\ldots,c_5).$}
	Model $\mathcal{M}_1$ is given by $c^{(1)} = (2,3,5,7,11,13)$ while
	model $\mathcal{M}_2$ is given by $c^{(2)} = (32,16,8,4,2,1).$
	We denote the respective likelihood functions by $L_u^{(1)}$ and~$L_u^{(2)}$ and the marginal likelihood integrals by 
	$$  \mathcal{I}_u^{(i)} \,\,\, = \,\,
	\int_{X_{>0}} \!\! L_u^{(i)}(x) \, \mu_{f,g} , \qquad i=1,2.$$
	In order to decide which model fits the data better, we compute the ratio 
	$K = \mathcal{I}_u^{(1)}/\mathcal{I}_u^{(2)}$ of the two marginal likelihood integrals. 
	This ratio is the {\em Bayes factor}.
	Using numerical cubature with tolerance {\tt 1e-5}, we find that $ \mathcal{I}_u^{(1)}\approx 5.675\cdot 10^{-38}$ and $ \mathcal{I}_u^{(2)}\approx 2.694 \cdot 10^{-39}.$ Therefore, $K\approx 21.06,$
	which reveals  that the model $\mathcal{M}_1$ is a better fit for $u$ than $\mathcal{M}_2.$
\end{example}

Another important Bayesian application is
sampling from the posterior distribution. In principle, this can be done with
\Cref{alg:mufg}, applied to the density $d_u$ in \eqref{eq:postprior}.
However, this fails to work as an off-the-shelf method. At present,
the method is  of theoretical interest only.
The challenge arises from large integer exponents, like $84$ in Equation \eqref{eq:Iu2}.
These exponents lead to a very low acceptance rate in
\Cref{prop:acceptancerate}. We
also observed this in practice: in a typical run of \Cref{alg:mufg}
for \eqref{eq:Iu1} with $N=100000,$ all samples are~rejected.

We conclude that our tropical sampling method
rests on solid and elegant mathematical foundations, and it
holds considerable promise for Bayesian inference. Yet,
more  research is needed to make it widely applicable for
computational statistics. For larger
sample size $U,$ the likelihood function $L_u$ has a sharp peak 
around its maximum, so it will be important to precompute
the critical points of  $L_u.$ The algebraic complexity for this task is the
ML degree of the model. This suggests combining
tropical sampling with the topological theory of ML degrees.
Our experiments also showed that exact symbolic algorithms
(cf.~\cite{LSZ}) are still surprisingly competitive. For instance, the
exact value of the integral $\mathcal{I}_u$ in
\eqref{eq:Iu1} equals

\vspace{-3mm}
{\footnotesize
\begin{align*} \frac{ 139123 \cdot 1256291317 \cdot 2602507379 \cdot 47336895027767486610187 } { 2^4 \cdot 3^6 \cdot 5^{88} \cdot 7^2 \cdot \! 11 \! \cdot 13^2 \cdot 17^2 \cdot 19 \cdot 23^2 \cdot 29^2 \cdot 31^2 \cdot 37^2 \cdot 41^2 \cdot 43 \cdot 47 \cdot 53 \cdot 59 \cdot 61 \cdot 67 \cdot 71 \cdot 79 \cdot 83}. \end{align*}
}
This rational number is intriguing.
We conclude that it would be desirable to combine the
methods of Sections~\ref{sec3} and~\ref{sec4} with such
exact evaluations. This is left for a future project.

\vspace*{0mm}
\section*{Acknowledgments}
We thank Thomas Lam for an insightful discussion on positive geometries, and two anonymous referees for their careful reading and 
helpful comments.

\vspace{-3mm}

\end{document}